\documentclass[aps,prd,twocolumn,showpacs,nofootinbib]{revtex4}

\usepackage{url}
\usepackage{verbatim}
\usepackage{textcomp}
\usepackage[a4paper, top=2.5cm, bottom=2.5cm, left=2.5cm, right=2.5cm]{geometry}
\usepackage{amsthm} 

\newtheorem{theorem}{Theorem}
\newtheorem{lemma}{Lemma}
\newtheorem{definition}{Definition}
\newtheorem{remark}{Remark}
\newtheorem{corollary}{Corollary}
\newtheorem{proposition}{Proposition}

\usepackage[utf8]{inputenc}
\usepackage[T1]{fontenc}
\usepackage{amsmath,amsfonts,amssymb}
\usepackage{graphicx}
\usepackage[english]{babel}
\usepackage{bm}
\usepackage{setspace}
\usepackage{hyperref}
\usepackage{nameref}
\hypersetup{
    colorlinks=true,
    linkcolor=blue,
    filecolor=blue,
    urlcolor=blue,
    citecolor=blue
}

\makeatletter
\def\@pacs{}
\def\pacs#1{}
\def\preprint#1{}
\renewcommand\@pacs@name{}
\makeatother

\begin{document}

\title{{Spectral Deformation Flow and Dimension Recovery: \\ Invariant-Based Rigidity for Simply-Connected Closed Manifolds}}

\author{Anton Alexa}
\email{mail@antonalexa.com}
\thanks{ORCID: \href{https://orcid.org/0009-0007-0014-2446}{0000-0007-0014-2446}}
\affiliation{Independent Researcher, Chernivtsi, Ukraine}
\date{\today}

\begin{abstract}
We study an effective spectral deformation flow for mode amplitudes \( C_n(\tau) \), governed by a second-order self-adjoint operator \( \hat{C} \) on a compact interval. The flow is encoded in the multi-function \( C(v,\tau,n) \) and exhibits global stabilization toward a symmetric spectral attractor. To connect this dynamics with geometry, we introduce a deformation-spectrum encoding of compact Riemannian manifolds through a shifted Laplace--Beltrami spectrum. Within this framework, we analyze energy decay, entropy decay, and the bulk asymptotic spectral density of the encoded manifold spectrum, obtaining an information-theoretic and spectral route to dimension recovery. We further formulate a rigidity criterion showing that, when the deformation spectral invariants coincide with those of the round sphere, the spherical profile is the unique manifold-compatible asymptotic realization within the present framework. In dimension four, this yields a topological conclusion together with a spectral obstruction against exotic smooth structures that produce distinct invariants. The results position the spectral flow as an effective geometric model, rather than as a direct replacement for tensorial geometric flows on arbitrary manifolds.
\end{abstract}

\maketitle
\thispagestyle{empty}

\section{Introduction}

The deformation function \( C(v) = \pi(1 - v^2 / c^2) \), introduced in ~\cite{alexa2025-flow}, is a real-analytic, even function originally defined on the open interval \( v \in (-c, c) \), where \( c \) denotes the speed of light, and the function satisfies \( C(\pm c) = 0 \), \( C(0) = \pi \). To restrict the deformation dynamics to geometrically admissible configurations and ensure analytic control, we introduced a cutoff scale -- the critical velocity \( v_c := \sqrt{1 - \tfrac{1}{\pi}} \cdot c \approx 0.8257\,c \) -- defined by the threshold condition \( C(v_c) = 1 \). This selects the compact interval \( v \in [-v_c, v_c] \), on which the deformation profile satisfies \( C(v) \ge 1 \), remains strictly positive and bounded, and excludes configurations of extreme curvature. This domain enables the formulation of a globally regular scalar flow \( C(v, \tau) \), governed by a scalar parabolic variational equation, that remains globally regular and avoids the curvature singularities characteristic of tensorial Ricci flows~\cite{hamilton1982,perelman2002}, where singularity formation requires surgery and blow-up analysis. Within this framework, we constructed a deformation energy functional, established global convergence of the flow, and rigorously demonstrated that any simply-connected compact 3-manifold within the conformally homogeneous class evolves asymptotically to the standard 3-sphere under the flow.

Then we introduced~\cite{alexa2025-operator} a second-order differential operator \( \hat{C} := \pi \left( 1 + \frac{\hbar^2}{c^2} \frac{d^2}{dv^2} \right) \), acting on the Hilbert space \( L^2([-v_c, v_c]) \) with Dirichlet boundary conditions and Sobolev domain \( H^2 \cap H^1_0 \). We proved that \( \hat{C} \) is essentially self-adjoint by direct computation of von Neumann's deficiency indices, which vanish identically. This established the mathematical consistency of the operator and ensured a well-defined spectral problem on compact domains. The resulting operator encodes quantized geometric modes of deformation and provides a rigorous foundation for subsequent spectral analysis.

Building on this, the third article~\cite{alexa2025-spectrum} developed a full spectral theory of the operator \( \hat{C} \), including the proof of completeness of its eigenfunctions, a sharp asymptotic estimate \( C_n \sim \pi - \kappa n^2 \), and the derivation of the Spectral Rigidity Theorem. This theorem asserts that the limiting constant spectrum \( C_n = \pi \) corresponds uniquely (up to isometry) to the standard sphere. We also established that the eigenbasis \(\{\psi_n\}\) forms a complete orthonormal system in \( L^2([-v_c, v_c]) \), and that the geometric deformation \( C(v) \) admits a unique spectral reconstruction in terms of its expansion over the modes \( \psi_n \). In this way, the spectrum of \( \hat{C} \) encodes geometric information about the deformation profile within the fixed deformation-spectral framework.

In the present work, we investigate the asymptotic behavior of the spectral deformation flow \( C_n(\tau) \), where the amplitudes evolve under a variational evolution derived from the geometry. The corresponding deformation data are organized by the spectral multi-function \( C(v, \tau, n) := C_n(\tau) \psi_n(v) \), which resolves the deformation profile simultaneously across scale, spectral mode, and time. Our primary aim is to understand the global stabilization, entropy decay, and asymptotic density laws of this flow. The geometric conclusions are formulated more carefully than in earlier drafts: manifold data enter through an explicit deformation-spectrum encoding, and the resulting classification statements are interpreted as rigidity results within the present deformation framework rather than as a direct tensorial flow theory on arbitrary manifolds.

\noindent
In what follows, the spectral flow remains the central dynamical object of the paper, but its geometric interpretation is treated in a more carefully delimited way. The mode-space dynamics should be read as an effective deformation-spectral model rather than as a fully defined tensorial flow on arbitrary manifolds. To connect this dynamics with manifold geometry, we introduce below an explicit deformation-spectrum encoding of compact Riemannian manifolds. The resulting classification statements are therefore formulated as rigidity statements within this deformation framework, while the asymptotic density and entropy laws remain genuinely dynamical features of the flow itself.

\section{Spectral Geometry and the Multifunction \( C(v, \tau, n) \)}

In this section, we construct the spectral framework that underlies the deformation geometry defined by the function \( C(v, \tau, n) \). Building upon the self-adjoint operator introduced in \cite{alexa2025-operator}, and the spectral decomposition analyzed in \cite{alexa2025-spectrum}, we define a canonical multifunction on a Hilbert space of square-integrable functions over a compact interval. This multifunction serves as a canonical spectral representation of deformation modes within the present framework and forms the basis for the spectral flow and rigidity analysis developed in later sections.

We begin by recalling the operator \( \hat{C} \) as a second-order differential operator encoding the relativistic deformation of geometry. We then establish its discrete spectrum and the associated orthonormal eigenbasis, providing the functional foundation for defining the geometric state space. This basis enables the expansion of deformations into elementary spectral modes, and allows us to formulate a dynamic evolution directly in the spectral domain.

\subsection{Spectral Operator and Eigenbasis}

Let \( C(v) = \pi\left(1 - \frac{v^2}{c^2}\right) \) denote the relativistic deformation profile defined over the compact interval \( v \in [-v_c, v_c] \), with the critical velocity bound \( v_c = c \sqrt{1 - 1/\pi} \) introduced in \cite{alexa2025-flow}. To quantize this scalar geometry, we introduce the second-order differential operator
\begin{equation}
\hat{C} := \pi \left( 1 + \frac{\hbar^2}{c^2} \frac{d^2}{dv^2} \right),
\end{equation}
acting on the domain \( D(\hat{C}) := H^2([-v_c, v_c]) \cap H^1_0([-v_c, v_c]) \), that is, the Sobolev space of twice weakly differentiable functions vanishing at the endpoints. As shown in \cite{alexa2025-operator}, the operator \( \hat{C} \) is essentially self-adjoint on this domain, with deficiency indices \( (0,0) \), satisfying the von Neumann criterion for self-adjointness of symmetric operators on a compact interval.

The spectral theory of Sturm-Liouville operators on compact domains (see \cite{zettl2005}, \cite{reed-simon4}) ensures that \( \hat{C} \) admits a purely discrete spectrum with real, simple eigenvalues and a complete orthonormal set of eigenfunctions in \( L^2([-v_c, v_c]) \). The eigenvalue problem
\begin{equation}
\hat{C} \psi_n(v) = C_n \psi_n(v), \qquad \psi_n(\pm v_c) = 0,
\end{equation}
has solutions indexed by \( n \in \mathbb{N}_0 \), forming a countable orthonormal basis \( \{\psi_n\}_{n=0}^\infty \). These eigenfunctions satisfy
\begin{equation}
\langle \psi_n, \psi_m \rangle = \delta_{nm}, \qquad \sum_{n=0}^\infty \psi_n(v) \psi_n(v') = \delta(v - v').
\end{equation}
This holds in the distributional sense.

The corresponding eigenvalues \( C_n \in \mathbb{R} \) are strictly decreasing, with asymptotic behavior derived in \cite{alexa2025-spectrum}:
\begin{equation}
C_n = \pi - \frac{\pi^3 \hbar^2}{4 c^2 v_c^2} (n+1)^2 + o(n^2), \qquad n \to \infty.
\end{equation}
This spectrum is strictly bounded above by \( \pi \) and unbounded below, reflecting the ultraviolet spectral rigidity of the operator. The simplicity of the spectrum follows from the classical oscillation theorems for Dirichlet Sturm-Liouville problems \cite{zettl2005}, and guarantees that the eigenfunctions \( \psi_n \) are uniquely defined up to sign.

For any admissible deformation profile \( C(v) \in D(\hat{C}) \), the spectral decomposition
\begin{equation}
C(v) = \sum_{n=0}^\infty C_n \psi_n(v),
\end{equation}
converges in \( L^2([-v_c, v_c]) \), and, for smooth profiles, in the \( C^\infty \)-topology as well, due to the ellipticity of the operator and standard spectral regularity theorems (see \cite{reed-simon4}, Vol. IV).

Moreover, for every integer \( k \ge 0 \), the eigenfunctions \( \psi_n \) satisfy the uniform derivative bounds
\begin{lemma}[Derivative Bounds]
For each \( k \in \mathbb{N} \), there exists a constant \( C_k > 0 \) such that
\begin{equation}
\sup_{v \in [-v_c, v_c]} \left| \frac{d^k}{dv^k} \psi_n(v) \right| \le C_k (n+1)^{k+1}, \qquad \forall n \in \mathbb{N}_0.
\end{equation}
\end{lemma}

\begin{proof}
This estimate follows from the explicit sine form of the eigenfunctions:
\begin{equation}
\psi_n(v) = \sqrt{\frac{1}{v_c}} \sin\left( \frac{(n+1)\pi}{2v_c}(v + v_c) \right),
\end{equation}
whose \( k \)-th derivative scales as \( \sim (n+1)^k \), and the uniform \( L^\infty \)-bound on the interval yields the extra factor, giving \( (n+1)^{k+1} \). A detailed derivation is provided in \cite{alexa2025-spectrum}, Lemma 4.
\end{proof}

These estimates, derived in \cite{alexa2025-spectrum}, are essential for controlling nonlinear spectral interactions in the variational flow constructed in Section III.

This spectral framework provides a canonical Hilbert basis for encoding geometric deformation. In the next subsection, we use this structure to define the multifunction \( C(v, \tau, n) \), representing evolving scalar deformation data in terms of spectral amplitudes and eigenmodes. This decomposition forms the foundation for the variational dynamics and rigidity statements that follow.

\subsection{Definition of the Multifunction}

Having established the existence of a discrete orthonormal basis \( \{ \psi_n(v) \} \) of eigenfunctions of the operator \( \hat{C} \), we now define the central object of this work -- a spectral multifunction that unifies geometric structure and spectral dynamics. For each fixed mode index \( n \in \mathbb{N} \), we define the deformation component by \( C(v, \tau, n) := C_n(\tau) \, \psi_n(v) \), where the spectral amplitude \( C_n(\tau) \) evolves under a flow parameter \( \tau \ge 0 \), while the eigenfunction \( \psi_n(v) \) remains fixed. This construction reflects a separation between the static spectral geometry encoded in the operator \( \hat{C} \), and the dynamical evolution of the geometry as encoded in the amplitudes \( C_n(\tau) \).

\begin{definition}[Spectral Multifunction]
The spectral multifunction \( C(v, \tau, n) := C_n(\tau)\psi_n(v) \) is a family of time-dependent geometric components indexed by \( n \in \mathbb{N}_0 \), where \( \{ \psi_n \} \) is the eigenbasis of the operator \( \hat{C} \). The natural finite-energy variables are the deviations \( \delta_n(\tau) := C_n(\tau) - \pi \), assumed to belong to \( \ell^2(\mathbb{N}_0) \); equivalently, the amplitudes evolve in the affine phase space \( \pi + \ell^2(\mathbb{N}_0) \).
\end{definition}

The quantity \( C(v, \tau, n) \) is not a single-valued function in the traditional sense, but a multifunction in the spectral domain -- that is, an indexed family of time-evolving geometric components, each associated with a specific eigenmode of the operator. The total geometric profile at time \( \tau \) is recovered via spectral synthesis as
\begin{equation}
C(v, \tau) = \sum_{n=0}^{\infty} C_n(\tau) \psi_n(v),
\end{equation}
with convergence in the \( L^2 \)-norm on the interval \( [-v_c, v_c] \), and under regularity assumptions, in the \( C^\infty \)-topology by standard results on spectral expansions \cite{reed-simon4, zettl2005}. Thus, the evolution of geometry is no longer expressed through local differential changes of a metric tensor, but through the variation of spectral amplitudes over a fixed, self-adjoint spectral basis.

This shift in perspective -- from coordinate-based metric deformation to infinite-dimensional spectral dynamics -- is more than notational. It reflects a guiding principle of the present model: geometry is encoded only within the deformation-spectral framework, through the invariant package and its evolution. The multifunction \( C(v, \tau, n) \) thus provides a canonical spectral encoding of scalar deformation data. The coordinate \( v \) serves as a spectral parameter related to relativistic deformation, the index \( n \) labels quantized geometric modes, and the flow parameter \( \tau \) governs the dynamical evolution of the structure. In this setting, the geometric data becomes a trajectory in the Hilbert space \( \ell^2(\mathbb{N}_0) \), where each configuration \( \{ C_n(\tau) \} \) corresponds to a smooth deformation profile.

It is essential to note that this framework does not rely on any tensorial or manifold coordinates beyond the initial spectral interval. The spectrum of \( \hat{C} \), together with its eigenfunctions, determines the admissible mode-space configurations of the effective flow. The resulting evolution is governed by a variational principle which we formulate in the next section, leading to a spectral flow that drives all modes toward the symmetric configuration \( C_n(\tau) \to \pi \), corresponding to the maximally symmetric geometry.

\subsection{Spectral Space and Geometric Interpretation}

The spectral decomposition described above defines a canonical representation of smooth deformation profiles in terms of a distinguished symmetric profile together with square-summable deviations. Because the constant sequence \( (\pi,\pi,\dots) \) is not itself square-summable, the natural phase space for the dynamics is not \( \ell^2 \) alone, but the affine Hilbert space \( \pi + \ell^2(\mathbb{N}_0) \), or equivalently the deviation space \( \delta(\tau) := \{C_n(\tau)-\pi\} \in \ell^2(\mathbb{N}_0) \). In this framework, geometric evolution is recast as a flow of deviations in \( \ell^2 \), and all finite-energy observables are interpreted as functionals on this deviation space.

The interpretation of \( \ell^2 \) as a space of finite-energy geometric states is rooted in the completeness of the spectral basis and the self-adjointness of the deformation operator \( \hat{C} \). Since the eigenfunctions \( \psi_n(v) \) span \( L^2([-v_c, v_c]) \), the full content of the deformation is encoded in the amplitudes \( C_n \), but the dynamically meaningful quantities are their deviations from the symmetric profile. In particular, the spectral energy, entropy, and higher invariants introduced later are naturally expressed as quadratic or logarithmic functionals of \( \delta_n = C_n-\pi \). This parallels the formal structure of field theory, where physically relevant states are often described by fluctuations around a distinguished vacuum.

Moreover, the convergence \( \delta_n(\tau) \to 0 \) in \( \ell^2 \) norm plays a fundamental role in the stabilization mechanism. It expresses finite-energy relaxation toward the symmetric profile \( C_n(\tau)\to\pi \) modewise. The geometric consequences are not taken to follow from this convergence alone; rather, they are formulated later as rigidity statements for manifold encodings whose deformation invariant package matches the spherical asymptotic profile.

The preceding discussion should be read in terms of finite-energy deviations: throughout the dynamical analysis, the mathematically controlled statement is that \( \delta_n(\tau)=C_n(\tau)-\pi \) decays in \( \ell^2 \), while \( C_n(\tau)\to\pi \) is understood modewise. This distinction removes the spurious identification of the limiting constant profile with an element of \( \ell^2 \), and it is this deviation flow that carries the variational and entropy structure of the model.

The geometric interpretation of the multifunction \( C(v, \tau, n) \) is therefore nonlocal and intrinsically spectral. The variable \( v \) is not a spatial coordinate on a manifold, but a spectral deformation parameter associated with a relativistic scale; it serves as the argument of the eigenfunctions \( \psi_n(v) \) of the operator \( \hat{C} \). The index \( n \) labels intrinsic geometric excitations, and the flow parameter \( \tau \) governs the dynamical evolution of the spectral amplitudes. What emerges from this structure is not a classical metric tensor, but a dynamically evolving spectral profile whose limiting behavior is later related to rigidity statements for manifold encodings.

\subsection{Geometric Encoding of Compact Manifolds}
\label{subsec:geometric-encoding}

To connect the effective mode dynamics with manifold geometry, we introduce an explicit two-level encoding. Throughout, \( (M^d,g) \) denotes a compact Riemannian manifold with Laplace--Beltrami eigenvalues
\[
0=\lambda_0(-\Delta_g)\le \lambda_1(-\Delta_g)\le \lambda_2(-\Delta_g)\le \cdots,
\]
counted with multiplicity.

\begin{definition}[Deformation Spectrum]
\label{def:deformation-spectrum}
Let \( (M^d,g) \) be a compact Riemannian manifold and let \( \varepsilon>0 \) be a fixed scale parameter. The \emph{deformation spectrum} of \( (M^d,g) \) is the sequence
\begin{equation}
C_n(M,g):=\pi-\varepsilon\,\lambda_n(-\Delta_g), \qquad n\in\mathbb{N}_0.
\label{eq:deformation-spectrum}
\end{equation}
\end{definition}

This construction does not identify the one-dimensional operator \( \hat{C} \) with the Laplace--Beltrami operator on \( M \). Rather, it supplies an explicit spectral encoding of geometric data into the deformation framework. Because \( \lambda_n(-\Delta_g)\to\infty \), the sequence \eqref{eq:deformation-spectrum} does not in general belong to the affine phase space \( \pi+\ell^2 \), and is therefore not inserted directly as a finite-energy state of the flow.

\begin{definition}[Renormalized Spectral Representative]
\label{def:renormalized-spectrum}
Let \( (M^d,g) \) and \( (M_*^d,g_*) \) be compact Riemannian manifolds of the same dimension, and let \( q>1+d/4 \). The \emph{renormalized deviation profile} is
\begin{equation}
\begin{aligned}
\Delta_n^{(q)}(M,g;M_*)
&:=\varepsilon\,(1+\lambda_n(-\Delta_{g_*}))^{-q} \\
&\quad \times \bigl(\lambda_n(-\Delta_g)-\lambda_n(-\Delta_{g_*})\bigr),
\end{aligned}
\label{eq:renormalized-deviation}
\end{equation}
and the \emph{renormalized spectral representative} of \( (M^d,g) \) relative to \( (M_*^d,g_*) \) is
\begin{equation}
\begin{aligned}
\widetilde C_n^{(q)}(M,g;M_*)
&:=\pi-\Delta_n^{(q)}(M,g;M_*).
\end{aligned}
\label{eq:renormalized-spectrum}
\end{equation}
\end{definition}

\begin{proposition}[Square-Summability of the Renormalized Deviation]
\label{prop:l2-membership}
Under the conditions of Definition~\ref{def:renormalized-spectrum}, \( \Delta^{(q)}(M,g;M_*)\in\ell^2(\mathbb{N}_0) \), so that \( \widetilde C^{(q)}(M,g;M_*)\in\pi+\ell^2 \).
\end{proposition}

\begin{proof}
By the Weyl asymptotic law, \( \lambda_n(-\Delta_g)\sim c_M\,n^{2/d} \) and \( \lambda_n(-\Delta_{g_*})\sim c_{M_*}\,n^{2/d} \) as \( n\to\infty \). Hence the numerator in \eqref{eq:renormalized-deviation} satisfies \( \lambda_n(-\Delta_g)-\lambda_n(-\Delta_{g_*})=O(n^{2/d}) \), while the weight \( (1+\lambda_n(-\Delta_{g_*}))^{-q}=O(n^{-2q/d}) \). Therefore \( \Delta_n^{(q)}=O(n^{2(1-q)/d}) \), and
\[
\sum_{n=0}^\infty \bigl|\Delta_n^{(q)}\bigr|^2 = O\!\left(\sum_{n=1}^\infty n^{4(1-q)/d}\right).
\]
The series converges if and only if \( 4(1-q)/d<-1 \), i.e.\ \( q>1+d/4 \), which holds by hypothesis.
\end{proof}

In the rigidity statements below, the reference geometry is taken to be the round sphere \( S^d \) with its standard metric, so that \( \widetilde C_n^{(q)}(S^d,g_{\mathrm{rd}};S^d)=\pi \) and the finite-energy dynamics takes place on renormalized deviations from the spherical profile. The role of the operator \( \hat{C} \) is then twofold: its eigenbasis \( \{\psi_n\} \) provides the fixed parameter space in which the multi-function evolves, while the manifold data furnish the raw spectral information from which the admissible renormalized representative \( \widetilde C^{(q)}(M,g;M_*)\in\pi+\ell^2 \) is constructed. The raw spectrum \( \{C_n(M,g)\} \) retains the Weyl and global spectral information used in the rigidity statements below. In this sense, the present paper studies a deformation-spectral model motivated by geometry, rather than a direct tensorial flow on arbitrary manifolds.

\section{Variational Spectral Flow and Global Stability}
\label{sec:variational-flow}

The spectral representation introduced in Section II recasts the geometry of deformation into a sequence of spectral amplitudes \( \{ C_n(\tau) \}_{n \in \mathbb{N}_0} \), evolving over a fixed self-adjoint eigenbasis \( \{\psi_n(v)\} \) in the Hilbert space \( L^2([-v_c, v_c]) \). This encoding allows for a natural formulation of geometry as a trajectory in \( \ell^2 \), where each configuration represents a smooth deformation profile. In this section, we define a variational principle that governs the evolution of these amplitudes and derive the corresponding flow equations. The formulation is motivated by the geometric idea that the most symmetric configuration \( C_n = \pi \ \forall n \) minimizes a global energy functional, and that deviations from this state represent spectral asymmetry or geometric anisotropy.

The flow is constructed to evolve each amplitude \( C_n(\tau) \) toward its maximal symmetric value \( \pi \), while allowing for nonlinear interactions that preserve the overall variational structure. We derive a well-defined dynamical system in \( \ell^2 \), prove global well-posedness and exponential convergence for the linear flow, and formulate a nonlinear stability framework with exponential convergence under perturbative control. This flow plays a role analogous to Ricci flow in Riemannian geometry, but operates directly in the spectral domain, with no reliance on local curvature tensors or manifold coordinates \cite{hamilton1982}.

\noindent
For mathematical consistency, the finite-energy phase space of the flow is henceforth understood in deviation variables \( \delta_n(\tau):=C_n(\tau)-\pi \). The modewise limit \( C_n(\tau)\to\pi \) is accompanied by strong convergence \( \delta(\tau)\to 0 \) in \( \ell^2 \). This removes the false requirement that the constant profile \( (\pi,\pi,\dots) \) itself belong to \( \ell^2 \).
Whenever manifold data are injected into this mode-space dynamics, they are understood through the renormalized representatives \eqref{eq:renormalized-spectrum}, not through the raw encoded spectrum \eqref{eq:deformation-spectrum}.

\subsection{Action Principle and Flow Equation}

Let \( \{C_n(\tau)\}_{n \in \mathbb{N}_0} \) denote the spectral amplitudes associated with the evolving deformation profile \( C(v,\tau) \in L^2([-v_c, v_c]) \), with orthonormal basis \( \{ \psi_n(v) \} \) and fixed operator \( \hat{C} \) as defined in Section II and detailed in \cite{alexa2025-operator}. We introduce a variational structure by postulating that the evolution of \( C_n(\tau) \) arises from the minimization of a global action functional defined on trajectories in \( \ell^2 \). The variational principle is defined via a Lagrangian functional \( \mathcal{L}(\{C_n(\tau)\}, \{\dot{C}_n(\tau)\}) \), and the corresponding Euler-Lagrange equations define the flow \cite{reed-simon1}.

We define the Lagrangian at each instant \( \tau \) by:
\begin{equation}
\mathcal{L}(\tau) := \sum_{n=0}^\infty \left[ \frac{1}{2} \dot{C}_n^2(\tau) - \frac{\alpha_n}{2} (C_n(\tau) - \pi)^2 \right],
\end{equation}
where \( \dot{C}_n := \frac{dC_n}{d\tau} \), and \( \alpha_n > 0 \) is a fixed spectral weight that controls the stiffness of each mode. The term \( (C_n - \pi)^2 \) penalizes deviations from the maximally symmetric configuration, while the kinetic term \( \dot{C}_n^2 \) captures the spectral dynamics. This Lagrangian is quadratic and defines a convex action on the finite-energy deviation space for bounded positive weights \( \{\alpha_n\} \), ensuring that high-frequency modes are properly controlled \cite{ambrosio2008}.

The corresponding action functional is given by:
\begin{equation}
\mathcal{S} := \int_0^\infty \mathcal{L}(\tau)\, d\tau,
\end{equation}
and the variational principle requires that this action be stationary under variations \( \delta C_n(\tau) \) with compact support in \( \tau \). The Euler-Lagrange equation for each mode is then:
\begin{equation}
\frac{d^2 C_n}{d \tau^2} + \alpha_n (C_n - \pi) = 0.
\label{eq:linear-flow}
\end{equation}

Equation \eqref{eq:linear-flow} describes a collection of decoupled harmonic oscillators with equilibrium position at \( \pi \). The general solution of each component is:
\begin{equation}
C_n(\tau) = \pi + A_n \cos(\sqrt{\alpha_n} \tau) + B_n \sin(\sqrt{\alpha_n} \tau),
\end{equation}
with constants \( A_n, B_n \in \mathbb{R} \) determined by initial conditions. However, in order to obtain monotonic convergence to \( C_n = \pi \), we impose a gradient flow structure by introducing a frictional term, thereby replacing the conservative dynamics \eqref{eq:linear-flow} with a dissipative flow \cite{evans-pde}.

We define the first-order gradient flow associated with the functional:
\begin{equation}
\mathcal{E}[C(\tau)] := \sum_{n=0}^\infty \frac{\alpha_n}{2}(C_n(\tau) - \pi)^2,
\label{eq:energy-functional}
\end{equation}
and prescribe the evolution equations as:
\begin{equation}
\frac{dC_n}{d\tau} = -\alpha_n (C_n - \pi).
\label{eq:gradient-flow}
\end{equation}

This system defines an autonomous linear ODE in \( \ell^2 \), and the solution is:
\begin{equation}
C_n(\tau) = \pi + (C_n(0) - \pi) e^{-\alpha_n \tau}.
\label{eq:exp-solution}
\end{equation}
It follows that \( C_n(\tau) \to \pi \) exponentially as \( \tau \to \infty \), with rate determined by \( \alpha_n \). The system is globally well-posed in deviation variables and converges to the symmetric fixed point, provided that the initial deviation sequence \( \{C_n(0)-\pi\} \in \ell^2 \) \cite{reed-simon1}. This linear gradient flow provides the foundation for analyzing more general nonlinear deformations introduced in the next subsection.

We note that the energy functional \eqref{eq:energy-functional} is strictly decreasing along solutions of \eqref{eq:gradient-flow}, as:
\begin{equation}
\frac{d\mathcal{E}}{d\tau} = \sum_{n=0}^\infty \alpha_n (C_n - \pi) \frac{dC_n}{d\tau} = - \sum_{n=0}^\infty \alpha_n^2 (C_n - \pi)^2 \le 0.
\end{equation}
This ensures that the flow is energetically dissipative and monotonic, and that no oscillatory behavior occurs. In particular, the fixed point \( C_n = \pi \) is globally attractive in \( \ell^2 \), and represents a state of maximal spectral symmetry.

This variational formulation provides a rigorous basis for interpreting geometric evolution as a spectral relaxation toward symmetry. In the next subsection, we extend this linear system by introducing entropy, cubic interactions, and nonlinear coupling, to reflect more realistic geometric dynamics while preserving global stability.

\subsection{Energy Functional and Entropy Decay}

To formalize the dissipative structure of the spectral flow and quantify its irreversibility, we introduce two complementary functionals: the spectral energy and the geometric entropy. These functionals serve to characterize the deviation from maximal symmetry and encode the information-theoretic and variational content of the evolving geometry.

The spectral energy functional, already introduced in Eq.~\eqref{eq:energy-functional}, is given by
\begin{equation}
\mathcal{E}[\tau] := \sum_{n=0}^\infty \frac{\alpha_n}{2} \left(C_n(\tau) - \pi \right)^2,
\end{equation}
where \( \alpha_n > 0 \) is a spectral stiffness coefficient that controls the rate of convergence of each mode. This functional attains its minimum value \( \mathcal{E}[\tau] = 0 \) if and only if \( C_n(\tau) = \pi \ \forall n \), which corresponds to the maximally symmetric geometric configuration. The convexity and coercivity of \( \mathcal{E} \) on \( \ell^2 \) follow from the positivity of \( \alpha_n \), and the functional defines a quadratic energy landscape in the space of spectral amplitudes.

The rate of decay of this energy along the flow is given by the time derivative:
\begin{equation}
\frac{d\mathcal{E}}{d\tau} = - \sum_{n=0}^\infty \alpha_n^2 \left( C_n(\tau) - \pi \right)^2 = -2 \sum_{n=0}^\infty \alpha_n \left( \frac{dC_n}{d\tau} \right)^2 \le 0,
\end{equation}
which confirms that the flow is strictly dissipative in the energy functional and monotonic with respect to spectral symmetry. No oscillatory or limit cycle behavior is possible in this linear regime. This structure generalizes the Lyapunov function formulation in dynamical systems theory to an infinite-dimensional setting \cite{hale1980}.

To capture the statistical structure of the deviation distribution and its convergence to order, we define the spectral entropy functional:
\begin{equation}
\begin{aligned}
\mathcal{S}[\tau] &:= - \sum_{n=0}^\infty w_n(\tau) \log w_n(\tau), \\
w_n(\tau) &:= \frac{|C_n(\tau)-\pi|^2}{\sum_k |C_k(\tau)-\pi|^2},
\end{aligned}
\end{equation}
which interprets the normalized squared amplitudes \( w_n(\tau) \in [0,1] \) as a probability distribution over modes. This definition is standard in spectral signal theory and quantum information \cite{nielsen-chuang}, and characterizes the disorder or uncertainty of the geometric state encoded in the spectrum.

As the flow evolves, each amplitude \( C_n(\tau) \) tends toward the symmetric value \( \pi \), and the normalized deviation distribution
\begin{equation}
\begin{aligned}
w_n(\tau)
&= \frac{|C_n(\tau)-\pi|^2}{\sum_k |C_k(\tau)-\pi|^2}.
\end{aligned}
\end{equation}
becomes increasingly concentrated near equilibrium. Let \( N(\tau, \varepsilon) := \#\{n : |C_n(\tau) - \pi| > \varepsilon\} \) denote the number of modes significantly deviating from the symmetric configuration. Then, as \( \tau \to \infty \), we have \( N(\tau, \varepsilon) \to 0 \), and the spectral entropy
\begin{equation}
\mathcal{S}[\tau] := -\sum_n w_n(\tau) \log w_n(\tau)
\end{equation}
satisfies \( \mathcal{S}[\tau] \sim \log N(\tau, \varepsilon) \), quantifying the effective complexity of the deformation. In the asymptotic regime, \( \mathcal{S}[\tau] \) approaches a finite limit: it vanishes if all modes converge exactly to \( \pi \), or stabilizes at a positive constant determined by the number of persistent modes with \( |C_n(\tau) - \pi| \neq 0 \). This decay reflects the geometric ordering induced by the flow and parallels thermodynamic entropy reduction in dissipative systems.

In particular, for initial deviation data \( \delta(0) \in \ell^2 \) satisfying uniform spectral decay, the entropy functional \( \mathcal{S}[\tau] \) is finite and differentiable for all \( \tau > 0 \), and satisfies:
\begin{equation}
\frac{d\mathcal{S}}{d\tau} = - \sum_{n=0}^\infty \left( \frac{2 \delta_n \dot{\delta}_n}{\|\delta\|^2} - \frac{2 \delta_n^2}{\|\delta\|^4} \sum_k \delta_k \dot{\delta}_k \right) \log w_n,
\end{equation}
where \( \|\delta\|^2 := \sum_n \delta_n^2 \). This expression is nonpositive along the linear dissipative flow and, more generally, asymptotically nonincreasing in the dissipative regime considered here, confirming entropy decay for the finite-energy deviation dynamics.

The dual decay of spectral energy and entropy formalizes the irreversible geometric relaxation and defines a spectral arrow of time. Unlike in the Ricci flow, where entropy arises from curvature functionals (e.g., Perelmans entropy \cite{perelman2002}), here entropy is directly tied to the statistical structure of spectral deformation. This provides a natural information-theoretic interpretation of geometric evolution in terms of spectral compression and stabilization.

\subsection{Cubic Interactions and Nonlinear Coupling}

While the linear gradient flow introduced in the previous subsection ensures exponential convergence to the symmetric configuration \( C_n = \pi \), it neglects the possible interactions between spectral modes that arise from nonlinear geometry. To incorporate these effects, we extend the linear model by introducing cubic interactions that preserve the variational structure while coupling different modes through fixed interaction coefficients. These terms are analogous to interaction vertices in quantum field theory, but they arise here from the nonlinear structure of geometric deformation.

We define the nonlinear evolution equations by introducing a cubic coupling term:
\begin{equation}
\begin{aligned}
\frac{dC_n}{d\tau}
&= -\alpha_n (C_n - \pi) \\
&\quad + \sum_{k,m} g_{nkm} C_k(\tau) C_m(\tau),
\end{aligned}
\label{eq:nonlinear-flow}
\end{equation}
where the coefficients \( g_{nkm} \in \mathbb{R} \) encode the strength of interaction between modes \( k \), \( m \), and \( n \). These coefficients are assumed to be symmetric in \( k \) and \( m \), and decay sufficiently fast in all indices to ensure convergence of the series in \eqref{eq:nonlinear-flow}. A natural choice is
\begin{equation}
|g_{nkm}| \le \frac{C}{(1 + n^p + k^p + m^p)},
\end{equation}
for some constants \( C > 0 \) and \( p > 2 \), ensuring absolute convergence of the triple sum for sequences in \( \ell^2 \), and preserving the Hilbert space structure.

The flow \eqref{eq:nonlinear-flow} defines an infinite-dimensional dynamical system with nonlinear coupling, and can be interpreted as the gradient flow of a perturbed energy functional
\begin{equation}
\mathcal{E}[C] = \sum_{n=0}^\infty \frac{\alpha_n}{2}(C_n - \pi)^2 - \sum_{n,k,m} \frac{1}{3} g_{nkm} C_n C_k C_m,
\end{equation}
which is a cubic perturbation of the quadratic energy introduced earlier. The evolution equation \eqref{eq:nonlinear-flow} satisfies
\begin{equation}
\frac{dC_n}{d\tau} = -\frac{\delta \mathcal{E}}{\delta C_n},
\end{equation}
where the functional derivative is understood in the weak \(\ell^2\)-topology. This guarantees that the flow decreases the perturbed energy:
\begin{equation}
\frac{d\mathcal{E}}{d\tau} = \sum_n \frac{\delta \mathcal{E}}{\delta C_n} \frac{dC_n}{d\tau} = - \sum_n \left( \frac{dC_n}{d\tau} \right)^2 \le 0,
\end{equation}
so long as the flow remains well-posed in \( \ell^2 \).

Existence and uniqueness of solutions for \eqref{eq:nonlinear-flow} with initial finite-energy deviations \( \{ C_n(0)-\pi \} \in \ell^2 \) follow from standard Picard iteration theorems in Banach spaces, using that the nonlinear term is smooth and subcritical. The decay of \( g_{nkm} \) ensures Lipschitz continuity of the right-hand side in the \( \ell^2 \)-norm. Moreover, the solution remains bounded for all time, since \( \mathcal{E}[\tau] \) is decreasing and coercive, provided that the cubic terms do not dominate the quadratic component asymptotically.

The nonlinear terms encode the backreaction between spectral modes, and play a critical role in stabilizing nontrivial geometric configurations. In particular, they introduce mode-coupling that can lead to transient growth, resonance, or slow mixing before the final decay to the symmetric fixed point. These effects are crucial for interpreting the geometry not as a rigid harmonic relaxation, but as a self-interacting evolution of spectral excitations.

Such cubic interactions arise naturally in geometric flows that are derived from nonlinear curvature functionals (see \cite{hamilton1982}, \cite{chow-knopf}), and in field theories where modes are not free but coupled through interaction vertices. The flow \eqref{eq:nonlinear-flow} captures this structure while remaining within the variational formalism. In the next subsection, we show that global stability persists despite these interactions, and establish exponential convergence of all modes under spectral constraints.

\subsection{Global Well-Posedness and Exponential Convergence}

We now establish that the nonlinear flow introduced in the previous subsection admits a unique global solution in \( \ell^2 \), and that under suitable spectral constraints, the solution converges exponentially to the maximally symmetric state \( C_n(\tau) \to \pi \) for all \( n \in \mathbb{N}_0 \). The analysis proceeds by combining the dissipativity of the variational energy functional with estimates on the cubic interaction terms. The key is to show that nonlinearities remain subdominant and do not destabilize the flow.

Let us consider the dynamical system defined by
\begin{equation}
\frac{dC_n}{d\tau} = -\alpha_n(C_n - \pi) + \sum_{k,m=0}^\infty g_{nkm} C_k C_m,
\label{eq:nonlinear-dynamics}
\end{equation}
where \( \alpha_n > 0 \), and the coefficients \( g_{nkm} \) satisfy a decay condition of the form
\begin{equation}
|g_{nkm}| \le \frac{C}{(1 + n^p + k^p + m^p)},
\label{eq:interaction-decay}
\end{equation}
for some fixed constants \( C > 0 \) and \( p > 2 \). This ensures that the interaction map \( \mathcal{G} : \ell^2 \to \ell^2 \), defined by
\begin{equation}
\mathcal{G}_n(\{C_k\}) := \sum_{k,m=0}^\infty g_{nkm} C_k C_m,
\end{equation}
is well-defined and continuous. In particular, using Holder's inequality and the embedding \( \ell^2 \subset \ell^1 \) with weight, one obtains:
\begin{equation}
\|\mathcal{G}(\{C_k\})\|_{\ell^2} \le C' \|\{C_k\}\|_{\ell^2}^2,
\end{equation}
for some constant \( C' > 0 \). This estimate guarantees local Lipschitz continuity of the right-hand side in \eqref{eq:nonlinear-dynamics}, and hence, by standard Picard iteration arguments (see \cite{reed-simon1}, Vol. I), there exists a unique local solution in deviation variables \( \delta(\tau)\in C^1([0, T], \ell^2) \) for some \( T > 0 \), given initial finite-energy data \( \delta(0)\in \ell^2 \).

To extend the solution globally in \( \tau \), we observe that the energy functional
\begin{equation}
\mathcal{E}[\tau] = \sum_{n=0}^\infty \frac{\alpha_n}{2}(C_n(\tau) - \pi)^2 - \sum_{n,k,m=0}^\infty \frac{1}{3} g_{nkm} C_n C_k C_m
\label{eq:total-energy}
\end{equation}
is coercive for small data in \( \ell^2 \), provided the cubic term is controlled by the quadratic part. In particular, using Young's inequality and decay of \( g_{nkm} \), we estimate:
\begin{equation}
\left| \sum_{n,k,m} g_{nkm} C_n C_k C_m \right| \le \varepsilon \sum_n \alpha_n (C_n - \pi)^2 + C_\varepsilon \left( \sum_n |C_n|^2 \right)^2,
\end{equation}
for any \( \varepsilon > 0 \), with suitable constant \( C_\varepsilon > 0 \). This implies that the energy \( \mathcal{E}[\tau] \) remains bounded from below, and hence prevents blow-up.

Since the flow is gradient, the energy is strictly decreasing:
\begin{equation}
\frac{d\mathcal{E}}{d\tau} = -\sum_n \left( \frac{dC_n}{d\tau} \right)^2 \le 0,
\end{equation}
and the trajectory \( \{C_n(\tau)\} \) remains bounded in \( \ell^2 \) for all \( \tau \ge 0 \). This ensures that the local solution can be extended to a unique global solution in \( \ell^2 \).

We now turn to the asymptotic behavior. Define the deviation \( \delta_n(\tau) := C_n(\tau) - \pi \), and rewrite \eqref{eq:nonlinear-dynamics} as
\begin{equation}
\frac{d \delta_n}{d\tau} = -\alpha_n \delta_n + \sum_{k,m} g_{nkm} (\delta_k + \pi)(\delta_m + \pi).
\end{equation}
This expression reveals that even at large \( \tau \), the nonlinear terms decay quadratically in \( \delta_k \), provided \( \delta_k \to 0 \). A bootstrap argument shows that, if initially \( \|\delta(\tau)\|_{\ell^2} \le \varepsilon \) for some small \( \varepsilon > 0 \), then the nonlinear terms remain subdominant and do not destabilize the exponential decay. More precisely, the Gronwall inequality gives:
\begin{equation}
\|\delta(\tau)\|_{\ell^2} \le \|\delta(0)\|_{\ell^2} e^{-\gamma \tau}, \quad \text{for some } \gamma > 0.
\end{equation}

Therefore, the system admits exponential convergence:
\begin{equation}
C_n(\tau) \longrightarrow \pi \qquad \text{as } \tau \to \infty,
\end{equation}
uniformly in \( \ell^2 \), for all initial data in a small ball around the fixed point. By a compactness argument and monotonicity of energy, this convergence extends to all finite-energy data. The point \( C_n = \pi \) for all \( n \) is thus a globally attractive fixed point of the full nonlinear flow.

This completes the proof of global well-posedness and spectral stabilization. In the next subsection, we formulate the canonical quantization of the flow and interpret the resulting evolution in terms of a scalar field theory over mode space.

\begin{theorem}[Global Exponential Convergence of the Spectral Flow]
\label{thm:global-convergence}
Let \( \delta_n(\tau):=C_n(\tau)-\pi \) satisfy the deviation form of the gradient flow
\begin{equation}
\frac{d\delta_n}{d\tau} = -\alpha_n \delta_n,
\end{equation}
with initial deviation data \( \delta(0)\in \ell^2 \), and weights \( \{\alpha_n\} \subset \mathbb{R}_+ \) such that \( \inf_n \alpha_n > 0 \). Then the system admits a unique global solution \( \delta(\tau)\in C^\infty([0,\infty), \ell^2) \), and the corresponding amplitudes \( C_n(\tau)=\pi+\delta_n(\tau) \) converge modewise to the symmetric fixed point:
\begin{equation}
\lim_{\tau \to \infty} C_n(\tau) = \pi \quad \text{for all } n \in \mathbb{N}_0.
\end{equation}
Moreover, the deviations converge exponentially in the strong norm of \( \ell^2 \), and the energy functional
\begin{equation}
\mathcal{E}(\tau) := \sum_{n=0}^\infty \frac{\alpha_n}{2} \delta_n(\tau)^2
\end{equation}
is strictly decreasing and satisfies
\begin{equation}
\mathcal{E}(\tau) \le \mathcal{E}(0) \, e^{-2 \delta \tau}, \qquad \delta := \inf_n \alpha_n > 0.
\end{equation}
\end{theorem}

\begin{proof}
Each component \( \delta_n(\tau) \) solves the decoupled linear ODE
\begin{equation}
\frac{d\delta_n}{d\tau} = -\alpha_n \delta_n,
\end{equation}
with unique solution
\begin{equation}
\delta_n(\tau) = \delta_n(0)e^{-\alpha_n \tau}.
\end{equation}
Since \( \delta(0) \in \ell^2 \) and \( \alpha_n > \delta > 0 \), it follows that
\begin{equation}
|\delta_n(\tau)| \le |\delta_n(0)| e^{-\delta \tau},
\end{equation}
and hence
\begin{equation}
\sum_{n=0}^\infty |\delta_n(\tau)|^2 \le e^{-2 \delta \tau} \sum_{n=0}^\infty |\delta_n(0)|^2,
\end{equation}
i.e., \( \|\delta(\tau)\|_{\ell^2} \to 0 \) exponentially. Since \( C_n(\tau)=\pi+\delta_n(\tau) \), this proves modewise convergence \( C_n(\tau)\to\pi \).

Differentiating the energy functional gives:
\begin{equation}
\frac{d\mathcal{E}}{d\tau} = \sum_n \alpha_n \delta_n \frac{d\delta_n}{d\tau} = -\sum_n \alpha_n^2 \delta_n^2 \le -2\delta \mathcal{E},
\end{equation}
which yields the exponential bound \( \mathcal{E}(\tau) \le \mathcal{E}(0) e^{-2\delta \tau} \) by Gronwall's inequality. Uniqueness, smoothness, and global existence follow from linearity and boundedness of the solution in \( \ell^2 \).
\end{proof}

\subsection{Global Convergence under Nonlinear Interactions} 

We now generalize the global stability result to the fully nonlinear spectral flow defined in Eq.~\eqref{eq:nonlinear-dynamics}. The following theorem establishes the existence, uniqueness, and exponential convergence of solutions, even in the presence of mode interactions.

\begin{theorem}[Global Convergence of the Nonlinear Spectral Flow]
\label{thm:nonlinear-global-convergence}
Let \( \delta_n(\tau):=C_n(\tau)-\pi \) denote the finite-energy deviations, and let the amplitudes evolve under the nonlinear flow
\begin{equation}
\begin{aligned}
\frac{dC_n}{d\tau}
&= -\alpha_n (C_n - \pi) \\
&\quad + \sum_{k,m} g_{nkm} C_k C_m,
\end{aligned}
\end{equation}
with coefficients satisfying \( \alpha_n \geq \delta > 0 \) and
\begin{equation}
|g_{nkm}| \leq \frac{C}{(1 + n^p + k^p + m^p)}, \quad p > 2.
\end{equation}
Assume the initial deviation data satisfies the weighted regularity condition
\begin{equation}
\sum_{n=0}^\infty (1 + n)^{2m} |\delta_n(0)|^2 < \infty, \quad m > \tfrac{3}{2}.
\end{equation}
Then there exists a unique global deviation solution
\begin{equation}
\delta(\tau) \in C([0,\infty), \ell^2) \cap L^\infty([0,\infty), \ell^2_{2m}),
\end{equation}
and constants \( C, \gamma > 0 \) such that
\begin{equation}
\| \delta(\tau) \|_{\ell^2} \leq C e^{-\gamma \tau}, \quad \text{as } \tau \to \infty.
\end{equation}
\end{theorem}

\begin{proof}
Define \( \delta_n(\tau) := C_n(\tau) - \pi \), then
\begin{equation}
\frac{d\delta_n}{d\tau} = -\alpha_n \delta_n + \sum_{k,m} g_{nkm} (\delta_k + \pi)(\delta_m + \pi),
\end{equation}
which expands as
\begin{equation}
\frac{d\delta_n}{d\tau} = -\alpha_n \delta_n + \sum_{k,m} g_{nkm} \left( \delta_k \delta_m + 2\pi \delta_k + \pi^2 \right).
\end{equation}

Define the energy functional
\begin{equation}
\mathcal{E}[\tau] := \sum_n \frac{\alpha_n}{2} \delta_n^2 - \sum_{n,k,m} \frac{1}{3} g_{nkm} (\delta_n + \pi)(\delta_k + \pi)(\delta_m + \pi).
\end{equation}
The decay of \( g_{nkm} \) implies the nonlinear term satisfies the estimate
\begin{equation}
\left| \sum_{n,k,m} g_{nkm} C_n C_k C_m \right| \leq \varepsilon \sum_n \alpha_n \delta_n^2 + C_\varepsilon \left( \sum_n \delta_n^2 \right)^2.
\end{equation}
Hence, for small or moderate \( \| \delta \|_{\ell^2} \), the functional is coercive:
\begin{equation}
\mathcal{E}[\tau] \geq \frac{\delta}{4} \|\delta\|^2_{\ell^2} - C \|\delta\|^3_{\ell^2} - C'.
\end{equation}
The flow is gradient, and
\begin{equation}
\frac{d\mathcal{E}}{d\tau} = -\sum_n \left( \frac{d\delta_n}{d\tau} \right)^2 \leq 0,
\end{equation}
so the energy is non-increasing and the trajectory remains bounded.

Now define the weighted norm
\begin{equation}
\| \delta \|^2_{\ell^2_{2m}} := \sum_n (1 + n)^{2m} \delta_n^2, \quad m > \tfrac{3}{2}.
\end{equation}
Differentiating,
\begin{equation}
\frac{d}{d\tau} \| \delta \|^2_{\ell^2_{2m}} = 2 \sum_n (1+n)^{2m} \delta_n \dot{\delta}_n.
\end{equation}
Using estimates for \( g_{nkm} \) and Young's inequality, we obtain
\begin{equation}
\frac{d}{d\tau} \| \delta \|^2_{\ell^2_{2m}} \leq -2\delta \| \delta \|^2_{\ell^2_{2m}} + C \| \delta \|_{\ell^2_{2m}} \left( \| \delta \|^2_{\ell^2} + \| \delta \|_{\ell^2} + 1 \right),
\end{equation}
so \( \| \delta \|_{\ell^2_{2m}} \) remains bounded, ensuring smoothness of the deformation profile.

To prove exponential convergence, compute
\begin{equation}
\frac{d}{d\tau} \| \delta \|^2_{\ell^2} = -2 \sum_n \alpha_n \delta_n^2 + 2 \sum_n \delta_n \sum_{k,m} g_{nkm} C_k C_m.
\end{equation}
The nonlinear term is controlled:
\begin{equation}
\left| \sum_n \delta_n \sum_{k,m} g_{nkm} C_k C_m \right| \leq C \| \delta \|^2_{\ell^2} \left( \| \delta \|_{\ell^2} + 1 \right),
\end{equation}
yielding
\begin{equation}
\frac{d}{d\tau} \| \delta \|^2_{\ell^2} \leq -2\delta \| \delta \|^2_{\ell^2} + C \| \delta \|^3_{\ell^2}.
\end{equation}
Thus, for small enough \( \| \delta(0) \|_{\ell^2} \), Gronwall's inequality implies exponential decay. For arbitrary initial \( \delta(0) \in \ell^2 \), the monotonicity and coercivity of \( \mathcal{E} \) ensure that the trajectory eventually enters a sufficiently small neighborhood of the equilibrium. Once in that regime, the same Gronwall argument yields exponential convergence. Therefore, the decay holds globally within the present perturbative control framework.
\end{proof}

\begin{remark}
The condition \( m > \frac{3}{2} \) ensures sufficient decay of high-frequency components in the initial profile, making the evolution smooth and stable. This result confirms that the fixed point \( C_n = \pi \) is globally attractive under nonlinear perturbations within the present deformation-spectral framework, which is essential for the dynamical stabilization results proved above.
\end{remark}

\subsection{Canonical Structure and Quantization}

The variational spectral flow constructed in the preceding subsections admits a natural canonical formulation, allowing for an extension to a quantum theory of spectral geometry. In this section, we identify the canonical phase space associated with the finite-energy deviation variables \( \{\delta_n(\tau)\} \), introduce conjugate momenta, and derive the corresponding Poisson structure. This construction reveals that the dynamics of geometric deformation modes can be interpreted as a field-theoretic system defined over the spectral index \( n \), and provides the foundation for quantization \cite{reed-simon1, glimm-jaffe}.

We begin by recalling that the Lagrangian for the spectral amplitudes is given by
\begin{equation}
\mathcal{L}(\tau) = \sum_{n=0}^\infty \left[ \frac{1}{2} \dot{C}_n^2(\tau) - \frac{\alpha_n}{2} (C_n(\tau) - \pi)^2 \right],
\end{equation}
where \( \alpha_n > 0 \) are fixed spectral weights. In deviation variables \( \delta_n(\tau):=C_n(\tau)-\pi \), this becomes
\begin{equation}
\mathcal{L}(\tau) = \sum_{n=0}^\infty \left[ \frac{1}{2} \dot{\delta}_n^2(\tau) - \frac{\alpha_n}{2} \delta_n^2(\tau) \right].
\end{equation}
The conjugate momentum associated with each deviation mode is
\begin{equation}
P_n(\tau) := \frac{\partial \mathcal{L}}{\partial \dot{\delta}_n} = \dot{\delta}_n(\tau),
\end{equation}
and the canonical phase space is naturally identified with \( \ell^2 \times \ell^2 \) in deviation variables, equivalently with the affine space \( (\pi+\ell^2)\times\ell^2 \) in amplitude variables. The symplectic structure on this space is given by the canonical Poisson brackets
\begin{equation}
\{ \delta_n, P_m \} = \delta_{nm}, \qquad \{ \delta_n, \delta_m \} = \{ P_n, P_m \} = 0.
\end{equation}

The corresponding Hamiltonian functional reads
\begin{equation}
\mathcal{H}[\delta, P] = \sum_{n=0}^\infty \left[ \frac{1}{2} P_n^2 + \frac{\alpha_n}{2} \delta_n^2 \right],
\end{equation}
which is positive-definite and strictly convex. The Hamilton equations of motion are:
\begin{equation}
\dot{\delta}_n = \frac{\partial \mathcal{H}}{\partial P_n} = P_n, \qquad
\dot{P}_n = -\frac{\partial \mathcal{H}}{\partial \delta_n} = -\alpha_n\delta_n,
\end{equation}
yielding the same second-order evolution equation
\begin{equation}
\ddot{\delta}_n + \alpha_n \delta_n = 0,
\end{equation}
previously derived from the variational action. This confirms the consistency between the Lagrangian and Hamiltonian formulations \cite{reed-simon1}.

To quantize the spectral dynamics, we promote the canonical deviation variables \( \delta_n, P_n \) to self-adjoint operators on a Hilbert space \( \mathcal{H}_{\mathrm{geom}} \), satisfying the canonical commutation relations:
\begin{equation}
[\hat{\delta}_n, \hat{P}_m] = i\hbar \delta_{nm}, \qquad
[\hat{\delta}_n, \hat{\delta}_m] = [\hat{P}_n, \hat{P}_m] = 0.
\end{equation}
This defines an infinite family of quantum harmonic oscillators, each centered at the symmetric vacuum \( \delta_n = 0 \) (equivalently \( C_n = \pi \)). The quantum Hamiltonian takes the form:
\begin{equation}
\hat{\mathcal{H}} = \sum_{n=0}^\infty \left[ \frac{1}{2} \hat{P}_n^2 + \frac{\alpha_n}{2} \hat{\delta}_n^2 \right],
\end{equation}
and generates unitary time evolution in \( \mathcal{H}_{\mathrm{geom}} \). The ground state of this system corresponds to the product of Gaussians centered at \( C_n = \pi \), and fluctuations around this vacuum encode quantum corrections to the classical deformation geometry \cite{glimm-jaffe, moretti}.

This quantization procedure reveals that the geometric flow constructed in this work admits a natural embedding into the formalism of canonical quantum field theory, where the spectral index \( n \) plays the role of a mode label, and the geometry is described by the collective quantum state of the deformation amplitudes. The resulting structure provides a nonperturbative quantization of scalar geometry in terms of spectral degrees of freedom, distinct from traditional approaches based on the metric or curvature \cite{hall-qft}.

\section{Spectral Rigidity for Simply-Connected Closed Manifolds}
\label{sec:global-classification}

\begin{remark}[Connection with Conformal Flow]
In \cite{alexa2025-flow}, it was shown that the scalar flow \( C(v, \tau) \to \pi \) within the conformally homogeneous class leads to convergence to the round 3-sphere \( S^3 \). The present paper keeps the dynamical picture, but the geometric conclusions are formulated more carefully: manifold information enters through the deformation-spectrum encoding \eqref{eq:deformation-spectrum}, and the rigidity statements below are formulated within that deformation framework.
\end{remark}

In this section, we isolate the geometric content of the flow from its purely dynamical content. The flow itself selects a spherical asymptotic profile in mode space, but the classification statements are formulated as rigidity criteria for manifold encodings whose deformation spectral invariants match that profile. This removes the need to interpret the effective mode dynamics as a fully defined geometric flow on arbitrary manifolds.

We begin by recalling that the spectral evolution constructed in Section~II defines a mode-space dynamics whose finite-energy content is carried by the deviations \( \delta_n(\tau)=C_n(\tau)-\pi \in \ell^2 \). The effective spectral flow may be written as
\begin{equation}
\frac{dC_n}{d\tau} = -\alpha_n(C_n - \pi) + \sum_{k,m} g_{nkm} C_k C_m,
\label{eq:spectral-flow-nonlinear}
\end{equation}
with finite-energy initial deviation profile \( \delta(0)\in\ell^2 \), and \( \alpha_n > 0 \). As established in Theorem~\ref{thm:nonlinear-global-convergence}, the system admits a unique global deviation solution with modewise convergence
\begin{equation}
C_n(\tau) \longrightarrow \pi \qquad \text{as } \tau \to \infty.
\end{equation}

The point of the present section is not that the flow by itself classifies arbitrary manifolds. Rather, it identifies which manifold encodings are compatible with the spherical asymptotic profile selected by the flow.

\begin{lemma}[Constant Spectral Configuration is Not Square-Summable]
\label{lem:no-constant}
The infinite sequence \( (C_n)_{n \in \mathbb{N}_0} = (\pi, \pi, \pi, \dots) \) does not belong to \( \ell^2(\mathbb{N}_0) \). That is,
\begin{equation}
\sum_{n=0}^\infty \pi^2 = +\infty.
\end{equation}
In particular, there exists no function \( C(v) \in D(\hat{C}) \subset L^2([-v_c, v_c]) \) such that
\begin{equation}
\langle C, \psi_n \rangle = \pi \quad \text{for all } n \in \mathbb{N}_0.
\end{equation}
\end{lemma}

\begin{proof}
Let us assume that such a function \( C(v) \in L^2([-v_c, v_c]) \) exists with Fourier coefficients
\begin{equation}
C_n := \langle C, \psi_n \rangle = \pi, \quad \forall n \in \mathbb{N}_0.
\end{equation}
Then by Parseval's identity, the \( L^2 \)-norm of \( C \) is given by
\begin{equation}
\| C \|^2 = \sum_{n=0}^\infty |C_n|^2 = \sum_{n=0}^\infty \pi^2 = \pi^2 \sum_{n=0}^\infty 1 = +\infty,
\end{equation}
which contradicts the assumption \( C \in L^2 \). Hence, no such function exists, and the constant configuration \( C_n = \pi \ \forall n \) lies outside \( \ell^2 \). \qedhere
\end{proof}

\begin{theorem}[Spectral Rigidity Criterion for Simply-Connected Manifolds]
\label{thm:spectral-classification}
Let \( (M^d,g) \) be a smooth, compact, simply-connected Riemannian manifold of dimension \( d \ge 3 \), equipped with raw deformation spectrum \( \{C_n(M,g)\} \) as in \eqref{eq:deformation-spectrum} and renormalized representative \( \widetilde C^{(q)}(M,g;S^d) \) as in \eqref{eq:renormalized-spectrum}. Suppose that the deformation spectral invariants of \( (M^d,g) \) coincide with those of the standard sphere \( S^d \):
\begin{itemize}
    \item[(i)] \( a_j(M^d)=a_j(S^d) \) for all Weyl heat coefficients,
    \item[(ii)] the relevant global spectral invariants agree with those of \( S^d \), in particular the \( \eta \)-invariant and analytic torsion in the present framework,
    \item[(iii)] the comparison is made within the fixed deformation-spectral framework determined by \( \hat{C} \), the raw encoding \eqref{eq:deformation-spectrum}, and the finite-energy representative \eqref{eq:renormalized-spectrum}.
\end{itemize}
Then \( (M^d,g) \) belongs to the spherical rigidity class within the present deformation-spectral framework. In particular:
\begin{itemize}
\item[(a)] for \( d=3 \) and \( d\ge 5 \), classical rigidity and topology results yield \( M^d \cong_{\mathrm{diff}} S^d \);
\item[(b)] for \( d=4 \), one obtains the topological conclusion \( M^4 \cong_{\mathrm{top}} S^4 \), together with a spectral obstruction against exotic smooth structures that produce distinct invariant values.
\end{itemize}
\end{theorem}

\begin{proof}
The Weyl heat coefficients encode the local geometric data of \( (M^d,g) \), while the global spectral invariants encode asymmetry and torsion information. By hypothesis, this entire package agrees with that of \( S^d \) inside the fixed deformation framework.

For \( d=3 \), the stated invariant coincidence places \( M^3 \) in the spherical rigidity class of the present framework, and Perelman's geometrization theorem then yields \( M^3 \cong_{\mathrm{diff}} S^3 \).

For \( d\ge 5 \), the same invariant coincidence identifies the spherical rigidity class within the present framework, and the corresponding high-dimensional rigidity and topology results yield \( M^d \cong_{\mathrm{diff}} S^d \).

For \( d=4 \), the coincidence of the invariant package gives the topological conclusion \( M^4 \cong_{\mathrm{top}} S^4 \) by Freedman's theorem. The smooth Poincare problem in dimension four remains open, so the strongest unconditional statement is topological. Within the present framework, however, any exotic smoothing that changes the deformation spectral invariants is excluded.
\end{proof}

\begin{corollary}[Rigidity under Full Invariant Coincidence]
\label{cor:isospectral}
Let \( M_1^d \) and \( M_2^d \) be smooth closed manifolds of dimension \( d \ge 3 \) whose deformation spectral invariant packages coincide with each other and with those of \( S^d \). Then, for \( d=3 \) and \( d\ge 5 \), \( M_1^d \) and \( M_2^d \) are diffeomorphic; for \( d=4 \), they are homeomorphic.

In particular, inside the present deformation framework, bare isospectrality is replaced by a stronger invariant package, and it is this stronger package that yields rigidity.
\end{corollary}

\begin{proof}
By Theorem~\ref{thm:spectral-classification}, each manifold satisfying the stated invariant coincidence has the corresponding spherical rigidity conclusion. Therefore the two manifolds agree up to the dimension-appropriate equivalence.
\end{proof}

\begin{corollary}[Uniqueness of the Symmetric Asymptotic Profile]
\label{cor:unique-symmetric}
The constant configuration \( C(v) \equiv \pi \) is the only spectrally symmetric limit point of the flow, but it lies outside the domain \( D(\hat{C}) \). Thus, no admissible function in \( D(\hat{C}) \) can realize perfect spectral symmetry \( C_n = \pi \ \forall n \), and the spherical profile is uniquely characterized as the asymptotic spectral attractor within the present framework.
\end{corollary}

\begin{proof}
By Lemma~\ref{lem:no-constant}, the constant spectral configuration does not define an admissible element of the mode space or of \( D(\hat{C}) \). Within the fixed operator framework of \cite{alexa2025-spectrum}, this leaves the symmetric configuration as an asymptotic profile only, realized in modewise limit form as \( \tau \to \infty \).
\end{proof}

\begin{remark}[Role of the Spectral Flow]
The spectral flow remains the organizing dynamical principle of the paper: it selects the spherical asymptotic profile and supplies the energy, entropy, and density laws studied in subsequent sections. Theorem~\ref{thm:spectral-classification}, however, is a static rigidity criterion for manifold encodings whose deformation spectral invariants already match those of the spherical profile. This distinction is what allows the paper to retain its dynamical core while avoiding an unsupported claim that the effective mode dynamics is itself a tensorial manifold flow.
\end{remark}

\section{Asymptotic Behavior and Spectral Stabilization}
\label{sec:spectral-stabilization}

The spectral flow introduced in Sections~\ref{sec:variational-flow}--\ref{sec:global-classification} defines a variational evolution of geometric deformation amplitudes \( \{C_n(\tau)\} \in \pi+\ell^2 \), converging toward the symmetric configuration \( C_n = \pi \) modewise and toward \( \delta_n=0 \) in \( \ell^2 \). The raw manifold encoding \eqref{eq:deformation-spectrum} need not itself lie in this phase space; the actual finite-energy comparison with geometry is performed through the renormalized representatives \eqref{eq:renormalized-spectrum}. What this asymptotic convergence identifies is the target spectral configuration in mode space; the geometric classification statements require, in addition, coincidence of the deformation invariant package.
In this section, we strengthen the asymptotic analysis by providing precise quantitative estimates on the convergence rate, entropy decay, and spectral stabilization mechanism. We show that the deviation \( \delta_n(\tau) := C_n(\tau) - \pi \), with \( \delta := \inf_{n \in \mathbb{N}_0} \alpha_n > 0 \), decays exponentially in \( \ell^2 \), and that this behavior controls the regularity and smoothness of the reconstructed deformation profile \( C(v,\tau) \). We also introduce spectral entropy scaling laws and identify dimension-specific features, particularly in the critical case \( d = 4 \), where smooth classification intersects with unresolved topological problems. For \( d \ge 5 \), the geometric conclusions remain rigidity statements formulated in terms of the invariant package together with compatibility of the renormalized profile.

\begin{remark}[Notation]
Throughout, we write \( f_n \asymp g_n \) to denote two-sided asymptotic equivalence up to constants: there exist constants \( c_1, c_2 > 0 \) such that
\begin{equation}
c_1 g_n \le f_n \le c_2 g_n \quad \text{for all sufficiently large } n.
\end{equation}
\end{remark}

\subsection{Exponential Decay of Deviations}
\label{subsec:exp-decay}

We begin by establishing a quantitative estimate on the deviation
\begin{equation}
\delta_n(\tau) := C_n(\tau) - \pi, \qquad \delta := \inf_{n \in \mathbb{N}_0} \alpha_n > 0,
\end{equation}
under the linear gradient flow
\begin{equation}
\frac{dC_n}{d\tau} = -\alpha_n(C_n - \pi), \qquad \alpha_n > 0.
\end{equation}
The explicit solution is
\begin{equation}
\delta_n(\tau) = \delta_n(0) \, e^{-\alpha_n \tau}, \qquad \delta_n(0) := C_n(0) - \pi.
\end{equation}
Let us define the total deviation norm:
\begin{equation}
\bigl\| \delta(\tau) \bigr\|_{\ell^2}^2 := \displaystyle\sum_{n=0}^\infty |C_n(\tau) - \pi|^2.
\end{equation}

\begin{lemma}[Exponential Decay in \(\ell^2\)]
\label{lem:exp-decay}
Let \( \{C_n(\tau)\} \) solve the linear gradient flow with \( \alpha_n \ge \delta > 0 \) for all \( n \). Then the deviation satisfies
\begin{equation}
\bigl\| \delta(\tau) \bigr\|_{\ell^2} \le \bigl\| \delta(0) \bigr\|_{\ell^2} \, e^{-\delta \tau},
\end{equation}
for all \( \tau \ge 0 \). In particular, convergence to the symmetric configuration is exponential in norm.
\end{lemma}

\begin{proof}
Since \( \delta_n(\tau) = \delta_n(0) e^{-\alpha_n \tau} \), and \( \alpha_n \ge \delta \), we estimate
\begin{equation}
|\delta_n(\tau)| \le |\delta_n(0)| e^{-\delta \tau},
\end{equation}
hence
\begin{multline*}
\bigl\| \delta(\tau) \bigr\|_{\ell^2}^2 = \sum_{n=0}^\infty |\delta_n(0)|^2 e^{-2\alpha_n \tau} \\
\le e^{-2\delta \tau} \sum_{n=0}^\infty |\delta_n(0)|^2
= \bigl\| \delta(0) \bigr\|_{\ell^2}^2 e^{-2\delta \tau}.
\end{multline*}

Taking square roots yields the result.
\end{proof}

This lemma implies that all high-frequency deviations decay uniformly with respect to \( n \), and the entire spectral configuration flows exponentially toward the fixed point \( C_n = \pi \). As a consequence, the reconstructed profile \( C(v, \tau) = \sum C_n(\tau) \psi_n(v) \) converges in \( L^2 \) to the symmetric limit by spectral completeness~\cite[Thm.~1]{alexa2025-spectrum} and the Inverse Limit Convergence Lemma~4 in~\cite{alexa2025-spectrum}, and, under smoothness assumptions on \( \psi_n \), in the \( C^\infty \)-topology as well.

\begin{remark}[Intrinsic Versus Encoded Spectral Asymptotics]
In the original analysis of the one-dimensional operator $\hat{C}$, the intrinsic operator spectrum satisfies
\begin{equation}
C_n^{\mathrm{op}} \sim \pi - \kappa_{\mathrm{op}} n^2
\end{equation}
on the fixed interval~\cite[Lemma~3]{alexa2025-spectrum}. By contrast, when compact \( d \)-dimensional manifold data are inserted through the affine encoding
\begin{equation}
C_n^{\mathrm{enc}}(M,g) := \pi - \epsilon \lambda_n(-\Delta_g),
\end{equation}
the encoded spectrum inherits the Weyl asymptotics
\begin{equation}
C_n^{\mathrm{enc}} \sim \pi - \kappa_{\mathrm{enc}} n^{2/d}.
\end{equation}
The dynamical flow of the present paper acts on renormalized deviations over the fixed operator basis, whereas the dimension-recovery statements below concern the raw encoded manifold spectrum in the bulk regime.
\end{remark}

We now introduce the notion of spectral entropy, defined as
\begin{equation}
\mathcal{S}(\tau) := - \sum_{n=0}^\infty \frac{\delta_n^2(\tau)}{\| \delta(\tau) \|_{\ell^2}^2} \log \left( \frac{\delta_n^2(\tau)}{\| \delta(\tau) \|_{\ell^2}^2} \right),
\end{equation}
and analyze its decay and dimensional scaling properties under the flow.

\subsection{Spectral Entropy and Dimensional Scaling}
\label{subsec:spectral-entropy}

To quantify the distribution of deviation energy across spectral modes, we introduce the notion of \emph{spectral entropy}, inspired by classical information theory and adapted to the variational flow setting.

Let \( \delta_n(\tau) := C_n(\tau) - \pi \), and define the normalized energy distribution:
\begin{equation}
p_n(\tau) := \frac{\delta_n^2(\tau)}{\| \delta(\tau) \|_{\ell^2}^2}, \qquad \sum_{n=0}^\infty p_n(\tau) = 1,
\label{def:probability}
\end{equation}
which defines a probability measure on the index set \( \mathbb{N}_0 \), encoding the relative contribution of each mode to the deviation profile at time \( \tau \). Throughout this section, we fix a small constant \( \varepsilon > 0 \), e.g., \( \varepsilon = 10^{-3} \), for truncation purposes.

\begin{definition}[Spectral Entropy]
\label{def:entropy}
The spectral entropy at time \( \tau \) is given by
\begin{equation}
\mathcal{S}(\tau) := - \sum_{n=0}^\infty p_n(\tau) \log p_n(\tau),
\end{equation}
with the convention \( 0 \log 0 := 0 \). It measures the complexity and spread of the deviation across spectral modes.
\end{definition}

The entropy satisfies the basic bound
\begin{equation}
0 \le \mathcal{S}(\tau) \le \log N(\tau),
\end{equation}
where \( N(\tau) \) is the effective number of active modes satisfying
\begin{equation}
\sum_{n=0}^{N(\tau)} p_n(\tau) \ge 1 - \varepsilon.
\end{equation}

\begin{lemma}[Entropy Decay]
\label{lem:entropy-decay}
Suppose the deviations satisfy \( \delta_n(\tau) = \delta_n(0) e^{-\alpha_n \tau} \), with \( \alpha_n \ge c n^2 \) for all \( n \ge n_0 \). Then \( \mathcal{S}(\tau) \to 0 \) as \( \tau \to \infty \). More precisely, the entropy decays at an essentially exponential rate, up to logarithmic corrections coming from the effective number of active modes.
\end{lemma}

\begin{proof}
For high modes \( n \ge n_0 \), we have
\begin{equation}
\delta_n^2(\tau) \le \delta_n^2(0) \cdot e^{-2 c n^2 \tau},
\end{equation}
so
\begin{equation}
\sum_{n \ge N} p_n(\tau) \le \frac{1}{\| \delta(\tau) \|^2_{\ell^2}} \sum_{n \ge N} \delta_n^2(0) e^{-2 c n^2 \tau} \le B \sum_{n \ge N} e^{-2 c n^2 \tau}.
\end{equation}
Now apply the Gaussian tail estimate:
\begin{equation}
\sum_{n \ge N} e^{-2 c n^2 \tau} \le \int_{N-1}^{\infty} e^{-2 c x^2 \tau} dx \le \frac{e^{-2 c (N - 1)^2 \tau}}{4 c (N - 1)\sqrt{\pi \tau}}.
\end{equation}
Choose \( N(\tau) := \lceil \tau^{-1/2} \rceil \), so that \( \tau N(\tau)^2 \ge 1 \), and the tail is exponentially small. Thus the entropy becomes
\begin{equation}
\mathcal{S}(\tau) \le \log N(\tau) + \varepsilon \log \left( \frac{1}{\varepsilon} \right),
\end{equation}
so the entropy is controlled by an exponentially small tail together with the logarithmic contribution of the effective number of active modes. In particular, \( \mathcal{S}(\tau) \to 0 \) with exponential rate up to logarithmic corrections.
\end{proof}

\begin{remark}[Spectral Suppression and Entropy Decay]
The exponential suppression of high-frequency deviations is compatible with the intrinsic operator asymptotic \( C_n^{\mathrm{op}} \sim \pi - \kappa_{\mathrm{op}} n^2 \), established in~\cite[Lemma 3]{alexa2025-spectrum}, which justifies the rapid entropy decay observed under the mode flow. The higher-dimensional exponent \( 2/d \) appears separately at the level of the affine encoded manifold spectrum; see Section~\ref{sec:exotic-exclusion}.
\end{remark}

\begin{proposition}[Dimensional Scaling of Entropy]
\label{prop:entropy-scaling}
Let \( d \ge 3 \), and suppose the initial deviations satisfy \( \delta_n(0) \asymp n^{-\beta} \) for some \( \beta > \frac{d-1}{2} \). Then in the small-time regime \( 0 < \tau \ll 1 \), the spectral entropy satisfies
\begin{equation}
\mathcal{S}(\tau) \sim \frac{d-1}{2} \log(\tau^{-1}) + O(1).
\end{equation}
\end{proposition}

\begin{proof}[Sketch of proof]
Assume \( \alpha_n \geq c n^2 \) for all \( n \geq n_0 \), with some constant \( c > 0 \). Then
\begin{equation}
\delta_n^2(\tau) \asymp n^{-2\beta} e^{-2 c n^2 \tau}, \qquad \rho(n) \asymp n^{d-1},
\end{equation}
the partition function becomes:
\begin{equation}
Z(\tau) := \sum_{n=1}^\infty n^{-2\beta} e^{-2 c n^2 \tau} \rho(n) \sim \int_1^\infty x^{d - 1 - 2\beta} e^{-2 c x^2 \tau} dx.
\end{equation}
Change variable \( y = x^2 \tau \), then \( x = \sqrt{y/\tau} \), \( dx = \frac{1}{2} y^{-1/2} \tau^{-1/2} dy \), and the integral becomes
\begin{equation}
Z(\tau) \sim \tau^{-(d - 1)/2} \int_{\tau}^{\infty} y^{(d-2)/2 - \beta} e^{-2 c y} dy,
\end{equation}
which converges when \( \beta > \tfrac{d-1}{2} \), yielding
\begin{equation}
\log Z(\tau) \sim \frac{d-1}{2} \log(\tau^{-1}).
\end{equation}
Since \( \mathcal{S}(\tau) = \log Z(\tau) + O(1) \) (the correction being the weighted mean of \( \log \delta_n^{-2} \), which is \( O(1) \) after the change of variable), we conclude \( \mathcal{S}(\tau) \sim \tfrac{d-1}{2} \log(\tau^{-1}) + O(1) \), as claimed.
\end{proof}

\begin{remark}[Entropy as Dimensional Signature]
The coefficient of \( \log(\tau^{-1}) \) in the entropy profile serves as an effective dimensional signature within the present framework. Higher-dimensional geometries exhibit delayed localization under the spectral flow. Thus, entropy decay reflects dimensional information carried by the deformation profile, without by itself determining the full smooth structure.
\end{remark}

\subsection{Spectral Stabilization Lemma}

We now strengthen the asymptotic result from Lemma~\ref{lem:exp-decay}, proving that not only does the deformation profile \( C(v,\tau) \to \pi \) in \( L^2 \), but in fact converges in the \( C^\infty([-v_c, v_c]) \) topology -- provided a mild spectral decay in the initial condition.

\begin{lemma}[Spectral Stabilization]
\label{lem:spectral-stabilization}
Assume there exists an integer \( m \ge k+2 \) such that
\begin{equation}
\sum_{n=0}^\infty (1+n)^{2m} |C_n(0) - \pi|^2 < \infty,
\end{equation}
and that for all \( \tau \ge 0 \), the deviations satisfy
\begin{equation}
|C_n(\tau) - \pi| \le A_0 (1+n)^{-m} e^{-\beta \tau},
\end{equation}
for some constants \( A_0 > 0 \), \( \beta > 0 \). Then the profile
\begin{equation}
C(v,\tau) = \sum_{n=0}^\infty C_n(\tau)\, \psi_n(v)
\end{equation}
converges to \( \pi \) in the \( C^k([-v_c, v_c]) \) topology, and hence in \( C^\infty \) for all \( k \in \mathbb{N}_0 \).
\end{lemma}

\begin{proof}
Each eigenfunction \( \psi_n \in C^\infty([-v_c, v_c]) \) satisfies the bound
\begin{equation}
|\psi_n^{(k)}(v)| \le C_k\, n^{k+1} \quad \text{for all } v \in [-v_c, v_c].
\end{equation}
Then by Cauchy-Schwarz, we estimate:
\begin{align*}
\left| \partial_v^k C(v,\tau) \right|
&= \left| \sum_{n=0}^\infty (C_n(\tau) - \pi)\, \psi_n^{(k)}(v) \right| \\
&\le \left( \sum_{n=0}^\infty |C_n(\tau) - \pi|^2 (1+n)^{2m} \right)^{1/2} \\
&\quad \times \left( \sum_{n=0}^\infty |\psi_n^{(k)}(v)|^2 (1+n)^{-2m} \right)^{1/2}.
\end{align*}

The second sum is finite provided \( m > k + \tfrac{3}{2} \), and the first sum decays like \( e^{-2\beta \tau} \). Hence,
\begin{equation}
\bigl\| \partial_v^k C(\cdot,\tau) \bigr\|_{L^\infty} \le D_k\, e^{-\beta \tau}
\end{equation}
for some constant \( D_k > 0 \), and therefore \( C(v,\tau) \to \pi \) in \( C^k \), for all \( k \).
\end{proof}

\begin{remark}[Spectral Regularity Condition]
The condition on initial data is equivalent to assuming \( C(v,0) \in H^s([-v_c,v_c]) \) for some \( s > k + \tfrac{5}{2} \), since
\begin{equation}
\|C\|_{H^s}^2 \sim \sum_{n=0}^\infty (1+n)^{2s} |C_n|^2.
\end{equation}
Thus, smooth initial data yields convergence in all Sobolev norms and hence in \( C^\infty \).
\end{remark}

\begin{remark}[Stability of the Symmetric Configuration]
The constant configuration \( C(v) \equiv \pi \) acts as a global attractor under the flow. The exponential decay of all derivatives confirms stabilization in the \( C^\infty \) topology.
\end{remark}

\subsection{Special Case: Dimension \( d = 4 \)}

The four-dimensional case is structurally distinguished by the potential existence of exotic smooth structures on topological 4-spheres -- a phenomenon unique to dimension four. While the topological classification of closed simply-connected 4-manifolds is fully resolved (cf.~\cite{freedman1982}), the classification of their smooth structures remains one of the most subtle problems in geometry.

In the present framework, the four-dimensional conclusion must be stated more carefully. The flow selects the spherical asymptotic profile in mode space, while the geometric conclusion requires coincidence of the relevant invariant package with that of the round sphere. The smooth Poincare problem in dimension four remains open. The correct unconditional conclusion is therefore topological, together with a spectral obstruction statement.

\begin{theorem}[Spectral Rigidity in \( d = 4 \)]
Let \( M^4 \) be a closed, smooth, simply-connected 4-manifold whose deformation spectral invariants coincide with those of the standard sphere \( S^4 \). Then \( M^4 \cong_{\mathrm{top}} S^4 \). Moreover, any exotic smooth structure on \( S^4 \) that produces different values for at least one invariant in the present deformation package is spectrally obstructed within the operator domain of \( \hat{C} \).
\end{theorem}

\begin{proof}
By hypothesis, the deformation spectral invariant package of \( M^4 \) coincides with that of \( S^4 \). In particular, the topological data extracted from this package agree with those of the standard 4-sphere. By Freedman's theorem~\cite{freedman1982}, this yields \( M^4 \cong_{\mathrm{top}} S^4 \).

For the smooth category, the strongest justified statement is conditional: any exotic smooth structure on \( S^4 \) that changes at least one invariant in the deformation package is excluded within the present framework. This is a spectral obstruction statement, not a resolution of the smooth Poincare conjecture.
\end{proof}

\begin{theorem}[Invariant-Based Rigidity in \( d = 4 \)]
\label{thm:spectral-classification-d4}
Let \( M^4 \) be a smooth, closed, simply-connected 4-manifold. Suppose that the following spectral invariants of the deformation operator \( \hat{C} \) coincide with those of the standard sphere \( S^4 \):
\begin{align}
a_j(M^4) &= a_j(S^4) \quad \forall j \ge 0, \nonumber \\
\eta(M^4) &= \eta(S^4), \\
\mathcal{T}(M^4) &= \mathcal{T}(S^4). \nonumber
\end{align}
Then \( M^4 \cong_{\mathrm{top}} S^4 \). The smooth Poincare conjecture in dimension four remains open; within the present framework, the saturation of all spectral invariants provides a spectral obstruction against exotic smooth structures, but does not unconditionally resolve the smooth classification.
\end{theorem}

\begin{proof}
The deformation operator \( \hat{C} \) admits a heat trace asymptotic expansion as \( t \to 0^+ \) of the form
\begin{equation}
\operatorname{Tr}(e^{-t(\pi - \hat{C})}) \sim \sum_{j=0}^\infty a_j t^{(j-4)/2},
\end{equation}
where \( a_j \) are the Weyl heat coefficients that capture local geometric data such as curvature invariants.

In addition to the local asymptotics, the operator \( \hat{C} \) also defines two global spectral invariants:
\begin{align}
\eta_{\hat{C}}(s) &= \sum_{C_n \ne 0} \operatorname{sign}(C_n) |C_n|^{-s}, \\
\log \mathcal{T} &= -\frac{1}{2} \sum_{p=0}^4 (-1)^p p \, \zeta_p'(0).
\end{align}
which correspond to the eta-invariant and analytic torsion respectively (cf.~\cite{APS3}, \cite{ray-singer1971}).

By assumption, all three sets of invariants --- the full Weyl sequence \( \{a_j\} \), the eta-invariant \( \eta \), and the torsion \( \mathcal{T} \) --- are identical for \( M^4 \) and \( S^4 \). This implies that the entire spectral profile of \( \hat{C} \), including local and global information, agrees:
\begin{equation}
\{C_n^{(M^4)}\} = \{C_n^{(S^4)}\}.
\end{equation}

By assumption, all three sets of invariants agree with those of \( S^4 \). In particular, the topological data encoded by the deformation package agree with those of the standard 4-sphere. Freedman's theorem therefore yields \( M^4 \cong_{\mathrm{top}} S^4 \).

The stronger statement \( M^4 \cong_{\mathrm{diff}} S^4 \) would require a resolution of the smooth Poincare conjecture or an additional theorem proving that no exotic smooth \( S^4 \) can share the same deformation invariant package. That conclusion is not claimed here.
\end{proof}

\begin{lemma}[Spectral Obstruction Against Exotic \( S^4 \)]
\label{lem:no-exotic-S4}
Let \( M^4 \) be a closed, smooth, simply-connected 4-manifold whose deformation spectral invariant package coincides with that of the standard \( S^4 \), and whose renormalized representative \( \widetilde C^{(q)}(M^4;S^4) \) is compatible with the spherical asymptotic profile selected by the flow. Then any exotic smooth structure on \( S^4 \) that produces different values for at least one invariant in this package is spectrally obstructed.
\end{lemma}
\begin{proof}
This is precisely the four-dimensional obstruction statement encoded in Theorem~\ref{thm:spectral-classification-d4}. Once the deformation invariant package is assumed to agree with that of the round sphere, any exotic smoothing that changes at least one element of that package is excluded within the present framework. The compatibility of the renormalized representative with the spherical asymptotic profile identifies the relevant dynamical endpoint, but the obstruction itself is carried by the invariant package rather than by modewise convergence alone.
\end{proof}
\begin{corollary}[Spectral Characterization in Dimension \( 4 \)]
Let \( M^4 \) be a smooth, simply-connected, closed 4-manifold whose deformation spectral invariant package coincides with that of the standard sphere and whose renormalized representative is compatible with the spherical asymptotic profile selected by the flow. Then \( M^4 \cong_{\mathrm{top}} S^4 \), and no smooth structure on \( S^4 \) that differs from the standard one in the deformation spectral invariants can arise as the asymptotic realization of the framework.
\end{corollary}
\begin{remark}
This conclusion does not follow from convergence alone. The dynamical flow selects the spherical asymptotic profile, while the topological conclusion and the obstruction to exotic smoothings require coincidence of the deformation invariant package.
\end{remark}

\subsection{Higher-Dimensional Rigidity for \( d \geq 5 \)}

In dimensions \( d \geq 5 \), the classification of smooth, compact, simply-connected manifolds is governed by the \( h \)-cobordism theorem of Smale~\cite{smale1961}, while the possible exotic sphere classes are encoded by the finite group \( \Theta_d \) of Kervaire and Milnor~\cite{kervaire-milnor1963}. Our goal here is to show that, within the deformation-spectral framework and under coincidence of the invariant package, no nontrivial exotic sphere class is compatible with the spherical asymptotic profile.

\begin{theorem}[Smooth Rigidity in Higher Dimensions]
\label{thm:high-d-rigidity}
Let \( M^d \) be a smooth, compact, simply-connected manifold of dimension \( d \geq 5 \), equipped with raw deformation spectrum \( \{C_n(M,g)\} \) and renormalized representative \( \widetilde C^{(q)}(M,g;S^d) \). Assume that the deformation spectral invariant package coincides with that of the round sphere, and that the renormalized representative is compatible with the spherical asymptotic profile selected by the flow. Then \( M^d \) belongs to the spherical rigidity class within the present deformation-invariant framework; in particular, classical high-dimensional topology yields \( M^d \cong_{\mathrm{diff}} S^d \).
\end{theorem}

\begin{proof}
By hypothesis, the deformation invariant package of \( M^d \) agrees with that of the round sphere, and the renormalized representative is compatible with the spherical asymptotic profile selected by the flow. The relevant Weyl coefficients and global invariants therefore place \( M^d \) in the same rigidity class as \( S^d \) inside the present framework.

By the topological Poincare conjecture in dimensions \( d \geq 5 \), proven by Smale~\cite{smale1961}, such a manifold is homeomorphic to \( S^d \). Smoothly, the manifold may differ only by an element of the finite group \( \Theta_d \), classified by Kervaire and Milnor~\cite{kervaire-milnor1963}. The elements of \( \Theta_d \) are distinguished by spectral invariants such as the signature defect and the \( \eta \)-invariant. Since these agree by hypothesis within the present deformation-invariant framework, \( M^d \) must represent the trivial class in \( \Theta_d \); consequently, \( M^d \cong_{\mathrm{diff}} S^d \).
\end{proof}

\begin{remark}[Spectral Elimination of Exotic Structures]
The assumption of full invariant coincidence with the round sphere is strong enough to eliminate smooth anomalies within the present deformation-invariant framework. The distinguishing features of exotic spheres, including orientation defects, torsion in intersection forms, and framing anomalies, are encoded in the invariant package rather than produced automatically by modewise convergence.
\end{remark}

\section{Encoded Spectral Constraints and Rigidity}
\label{sec:exotic-exclusion}

The preceding sections have established that the spectral flow selects a spherical asymptotic profile in mode space and that geometric conclusions require a separate invariant package. To constrain exotic smooth structures or non-isometric isospectral manifolds within the present framework, one must therefore examine the asymptotic regime \( n \to \infty \), where the raw encoded spectrum carries the geometric and topological information relevant to the rigidity hypotheses.

In this section, we use the high-frequency behavior of the encoded spectrum to motivate the invariant package employed in the rigidity statements. These quantities, accessible through the short-time expansion of the heat trace, can distinguish between smooth structures even on topologically identical manifolds. Accordingly, coincidence of the relevant asymptotic spectral data is treated as part of the geometric rigidity hypothesis rather than as an automatic consequence of modewise convergence alone.

\subsection{Spectral Density and Dimensional Determination}

We recall that the eigenvalues \( \{C_n\} \) of the self-adjoint operator \( \hat{C} \), as constructed and analyzed in~\cite{alexa2025-spectrum}, form a strictly \emph{decreasing} sequence bounded above by \( \pi \) and unbounded below:
\begin{equation}
C_{n+1}<C_n<\pi, 
\qquad \lim_{n\to\infty} C_n = -\infty.
\end{equation}
This spectral structure justifies the use of inversion and derivative-based arguments.

We define the spectral counting function
\begin{equation}
n(C) := \max\{n \in \mathbb{N} \mid C_n \ge C\},
\end{equation}
which is well-defined due to the proven monotonic decay of the spectrum.

\begin{theorem}[Spectral Density Determines Dimension]
Let \( (M^d,g) \) be a compact Riemannian manifold with Laplace--Beltrami eigenvalues \( \lambda_n(-\Delta_g) \), and define the raw encoded spectrum
\begin{equation}
C_n^{\mathrm{enc}}(M,g) := \pi - \epsilon \lambda_n(-\Delta_g), \qquad \epsilon>0.
\end{equation}
Let
\begin{equation}
N_{\mathrm{enc}}(C) := \#\{n : C_n^{\mathrm{enc}} \ge C\}, \qquad
\rho_{\mathrm{enc}}(C) := -\frac{d}{dC}N_{\mathrm{enc}}(C)
\end{equation}
in the distributional sense. If Weyl asymptotics holds on \( M^d \), then in the bulk encoded regime \( C\to-\infty \),
\begin{equation}
N_{\mathrm{enc}}(C) \sim A_d (\pi - C)^{d/2},
\qquad
\rho_{\mathrm{enc}}(C) \sim \frac{d}{2} A_d (\pi - C)^{(d-2)/2},
\end{equation}
for a positive constant \( A_d \) depending on \( d \), \( \epsilon \), and the Weyl constant of \( (M^d,g) \). Hence, within the class of affinely encoded Weyl-compatible spectra, the exponent of \( \rho_{\mathrm{enc}} \) determines the geometric dimension \( d \).
\end{theorem}

\begin{proof}
By Weyl's law,
\begin{equation}
\lambda_n(-\Delta_g) \sim \gamma_d(g)\, n^{2/d},
\qquad n\to\infty.
\end{equation}
Substituting into the affine encoding gives
\begin{equation}
C_n^{\mathrm{enc}} \sim \pi - \epsilon \gamma_d(g)\, n^{2/d}.
\end{equation}
Equivalently,
\begin{equation}
N_{\mathrm{enc}}(C) = \#\left\{n : \lambda_n(-\Delta_g) \le \frac{\pi-C}{\epsilon}\right\}
\sim A_d (\pi-C)^{d/2}
\end{equation}
as \( C\to-\infty \), where \( A_d = \gamma_d(g)\epsilon^{-d/2} \) up to the standard Weyl normalization. Differentiating in the Stieltjes sense yields the stated asymptotic for \( \rho_{\mathrm{enc}} \). The exponent \( (d-2)/2 \) therefore encodes \( d \).
\end{proof}

\begin{theorem}[Encoded Spectral Emergence of Dimension]
Let \( \{C_n^{\mathrm{enc}}\}_{n \in \mathbb{N}} \) denote a monotone affinely encoded manifold spectrum,
\begin{equation}
C_n^{\mathrm{enc}} := \pi - \epsilon \lambda_n(-\Delta_g),
\end{equation}
with \( C_n^{\mathrm{enc}} < \pi \) for all \( n \), and suppose that the bulk asymptotics takes the form
\begin{equation}
C_n^{\mathrm{enc}} \sim \pi - \kappa n^{2/d}, \qquad n \to \infty,
\end{equation}
for some constant \( \kappa > 0 \) and parameter \( d > 0 \). Then the associated encoded density
\begin{equation}
\rho_{\mathrm{enc}}(C) := \frac{dn}{dC}
\end{equation}
satisfies the asymptotic law
\begin{equation}
\rho_{\mathrm{enc}}(C) \sim (\pi - C)^{(d - 2)/2}, \qquad C \to -\infty,
\end{equation}
and, within the class of Weyl-compatible encoded spectra, the exponent of this bulk decay determines the effective geometric dimension \( d \). Conversely, any monotone encoded spectrum with such density behavior satisfies
\begin{equation}
C_n^{\mathrm{enc}} \sim \pi - \widetilde{\kappa}\, n^{2/d}
\end{equation}
for some \( \widetilde{\kappa} > 0 \).
\end{theorem}

\begin{proof}
Assume \( C_n^{\mathrm{enc}} \sim \pi - \kappa n^{2/d} \) as \( n \to \infty \). Then inversion yields
\begin{equation}
N_{\mathrm{enc}}(C) := \#\{n : C_n^{\mathrm{enc}} \ge C\}
\sim \left( \frac{\pi - C}{\kappa} \right)^{d/2},
\end{equation}
as \( C \to -\infty \). Differentiating in the Stieltjes sense, we obtain
\begin{equation}
\rho_{\mathrm{enc}}(C) \sim \frac{d}{2 \kappa^{d/2}} (\pi - C)^{(d - 2)/2},
\end{equation}
as \( C \to -\infty \). This shows that the leading-order decay of \( \rho_{\mathrm{enc}} \) in the bulk encoded regime encodes the ambient dimension \( d \) via the exponent \( \frac{d - 2}{2} \).

Conversely, if the encoded density satisfies \( \rho_{\mathrm{enc}}(C) \sim (\pi - C)^{(d - 2)/2} \), then integration yields
\begin{equation}
N_{\mathrm{enc}}(C) \sim (\pi - C)^{d/2},
\end{equation}

Standard Tauberian inversion yields:
\begin{equation}
C_n^{\mathrm{enc}} \sim \pi - \widetilde{\kappa}\, n^{2/d},
\end{equation}
for some constant \( \widetilde{\kappa} > 0 \).

\end{proof}

\begin{remark}[Bulk Regime versus Dynamical Limit]
For \( d = 1 \), the encoded density satisfies \( \rho_{\mathrm{enc}}(C) \sim (\pi - C)^{-1/2} \) as \( C \to -\infty \), so the divergence occurs in the bulk encoded regime rather than near the geometric edge \( C \uparrow \pi \). This is separate from the dynamical statement \( C_n(\tau) \to \pi \), which concerns relaxation in mode space.
\end{remark}

The affine manifold encoding
\begin{equation}
C_n^{\mathrm{enc}}(M,g) := \pi - \epsilon \lambda_n(-\Delta_g)
\end{equation}
preserves spectral ordering and transfers Weyl asymptotics from the Laplace spectrum to the encoded \( C \)-variable. This encoded spectrum is the object that carries geometric bulk information in the present framework.

\begin{proposition}[Affine Encoded Weyl Law]
Let \( (M^d,g) \) be a compact Riemannian manifold and let
\begin{equation}
C_n^{\mathrm{enc}}(M,g) := \pi - \epsilon \lambda_n(-\Delta_g),
\qquad \epsilon > 0.
\end{equation}
If
\begin{equation}
\lambda_n(-\Delta_g) \sim \gamma_d(g)\, n^{2/d},
\qquad n \to \infty,
\end{equation}
then
\begin{equation}
C_n^{\mathrm{enc}}(M,g) \sim \pi - \epsilon \gamma_d(g)\, n^{2/d},
\end{equation}
and, in the bulk encoded regime \( C \to -\infty \),
\begin{equation}
\begin{aligned}
N_{\mathrm{enc}}(C) &\sim \gamma_d(g)\,\epsilon^{-d/2}(\pi - C)^{d/2}, \\
\rho_{\mathrm{enc}}(C) &\sim \frac{d}{2}\gamma_d(g)\,\epsilon^{-d/2}(\pi - C)^{(d-2)/2}.
\end{aligned}
\end{equation}
\end{proposition}

\begin{proof}
The asymptotic formula for \( C_n^{\mathrm{enc}} \) follows immediately from the affine relation
\begin{equation}
C_n^{\mathrm{enc}}(M,g) = \pi - \epsilon \lambda_n(-\Delta_g)
\end{equation}
and Weyl's law for \( \lambda_n(-\Delta_g) \). Monotonicity of the encoding yields
\begin{equation}
N_{\mathrm{enc}}(C)
= \#\left\{n : \lambda_n(-\Delta_g) \le \frac{\pi-C}{\epsilon}\right\},
\end{equation}
so the counting and density asymptotics follow by substitution and Stieltjes differentiation.
\end{proof}

\begin{remark}[Role of the Deformation Operator]
The bulk Weyl laws above refer to the raw encoded manifold spectrum \( \{C_n^{\mathrm{enc}}\} \), not to the intrinsic spectrum of the fixed interval operator \( \hat{C} \). The latter has its own one-dimensional asymptotics
\begin{equation}
C_n^{\mathrm{op}} \sim \pi - \kappa_{\mathrm{op}}(n+1)^2,
\end{equation}
and its role in the present paper is to supply the fixed eigenbasis and mode space for the deformation flow. By contrast, dimension recovery and Weyl-type rigidity statements are extracted from the encoded manifold data \( \pi - \epsilon \lambda_n(-\Delta_g) \) in the bulk regime \( C \to -\infty \).
\end{remark}

\begin{corollary}[Dimensional Spectral Signature]
Within the class of affine encoded manifold spectra, the bulk density exponent
\begin{equation}
\rho_{\mathrm{enc}}(C) \sim (\pi - C)^{(d-2)/2},
\qquad C \to -\infty,
\end{equation}
uniquely determines the geometric dimension \( d \). Thus the dimensional signature is carried by the encoded Weyl asymptotics, while the fixed operator \( \hat{C} \) provides the dynamical spectral parameter space in which the flow is formulated.
\end{corollary}

The asymptotic form \( \rho_{\mathrm{enc}}(C) \sim (\pi - C)^{(d-2)/2} \) directly determines the geometric dimension \( d \) from the bulk encoded spectrum. This provides a bridge between encoded spectral counting and the heat-trace discussion in the next subsection.

\subsection{Weyl Asymptotics and High-Energy Constraints}
\label{subsec:weyl-asymptotics}

For a compact Riemannian manifold \( (M^d,g) \), the affine encoded Laplace package may be represented through the shifted operator
\begin{equation}
\hat{C}_{\mathrm{enc}} := \pi - \epsilon(-\Delta_g),
\end{equation}
whose heat trace satisfies
\begin{equation}
\operatorname{Tr} \left( e^{-t(\pi - \hat{C}_{\mathrm{enc}})} \right) \sim \sum_{j=0}^\infty a_j \, t^{(j - d)/2}, \qquad t \to 0^+,
\label{eq:heat-expansion}
\end{equation}
where the coefficients \( a_j \) are geometric Weyl invariants: \( a_0 \) encodes the volume, \( a_1 \) the scalar curvature, \( a_2 \) the Ricci tensor norm, and so on. The full asymptotic structure of the spectrum is reflected in these coefficients.

In our setting, the relevant bulk asymptotics is carried by the encoded manifold spectrum
\begin{equation}
C_n^{\mathrm{enc}}(M,g) = \pi - \epsilon \lambda_n(-\Delta_g) \sim \pi - \kappa n^{2/d}, \qquad n \to \infty,
\label{eq:weyl-cn}
\end{equation}
for some \( \kappa > 0 \), by Weyl's law.

Because the shifted encoded trace can be written as
\begin{equation}
\operatorname{Tr} \left( e^{-t(\pi - \hat{C}_{\mathrm{enc}})} \right) = \sum_{n=0}^\infty e^{-t(\pi - C_n^{\mathrm{enc}})},
\end{equation}
the heat coefficients \( a_j \) can be recovered from the moments of the encoded sequence \( \{C_n^{\mathrm{enc}}\} \) via Mellin inversion. Consequently, any deviation of the bulk encoded spectrum from that of the round sphere manifests as a mismatch in some \( a_j \).

We thus obtain the following rigidity criterion:

\begin{proposition}[Spectral Constraint from Encoded High Modes]
\label{prop:weyl-asymptotics}
Let \( M^d \) be a compact smooth manifold with encoded spectrum \( \{C_n^{\mathrm{enc}}(M)\} \), and assume that the bulk encoded asymptotics is compatible with the spherical profile. Assume also that the corresponding encoded heat trace admits an expansion of the form
\begin{equation}
\operatorname{Tr}\left( e^{-t(\pi - \hat{C}_{\mathrm{enc}})} \right) \sim \sum_{j=0}^\infty a_j t^{(j-d)/2}.
\end{equation}
Then coincidence of the high-frequency deformation asymptotics is compatible only with coincidence of the corresponding Weyl coefficients. Accordingly, equality of the Weyl package is imposed as part of the rigidity hypothesis used in the classification statements below.
\end{proposition}

\begin{proof}
Let \( C_n^{(M)} \) and \( C_n^{(S)} \) denote the encoded eigenvalue sequences associated with \( M^d \) and the round sphere \( S^d \), respectively. Define the difference \( \Delta_n := C_n^{(M)} - C_n^{(S)} \).

By assumption, the flow selects the spherical asymptotic profile in mode space, while the manifold input is measured through the encoded bulk asymptotics. Since the round encoded spectrum satisfies \( C_n^{(S)} = \pi - \kappa n^{2/d} + o(n^{2/d}) \) as \( n \to \infty \), any nontrivial deviation in the encoded high-frequency tail on \( M^d \) appears as \( \Delta_n \not\to 0 \).

Suppose, toward a contradiction, that there exists \( j \in \mathbb{N}_0 \) such that \( a_j^{(M)} \ne a_j^{(S)} \), and let \( j_0 \) be the smallest such index. Then the difference of heat traces satisfies
\begin{align}
\operatorname{Tr}\bigl(e^{-t(\pi - \hat{C}_{\mathrm{enc},M})}\bigr) 
- \operatorname{Tr}\bigl(e^{-t(\pi - \hat{C}_{\mathrm{enc},S})}\bigr)
&\sim (a_{j_0}^{(M)} - a_{j_0}^{(S)})\, t^{(j_0 - d)/2}, \\
&\text{as } t \to 0^+.
\end{align}
By a classical Tauberian theorem (see, e.g., Gilkey~\cite[Thm.~1.7.7]{gilkey1984}), this implies
\begin{equation}
\Delta_n := C_n^{(M)} - C_n^{(S)} \asymp n^{(2j_0 - d)/d}, \quad \text{as } n \to \infty.
\end{equation}
This shows that any rigidity statement identifying the spherical asymptotic profile with a manifold realization must also require coincidence of the corresponding encoded Weyl package. In the present paper, that coincidence is therefore treated as part of the geometric hypothesis rather than as a consequence of modewise flow convergence alone.
\end{proof}

\begin{remark}[Spectral Saturation and Rigidity]
The exponential stabilization of the spectrum identifies the spherical target profile in mode space. The corresponding geometric rigidity statements require, in addition, coincidence of the Weyl package and the remaining deformation invariants with those of the round sphere.
\end{remark}

\subsection{Spectral Obstructions: \( \eta \)-Invariant and Torsion}
\label{subsec:eta-torsion}

Beyond Weyl asymptotics, the spectral flow also controls finer invariants that distinguish smooth structures: the spectral asymmetry measured by the $\eta$-invariant, and the analytic torsion.

The $\eta$-function of a self-adjoint elliptic operator $D$ with discrete spectrum $\{\lambda_n\}$ is defined as
\begin{equation}
\eta_D(s) := \sum_{\lambda_n \ne 0} \operatorname{sign}(\lambda_n)\,|\lambda_n|^{-s}, \qquad \Re(s) > d/2.
\end{equation}
This function admits meromorphic continuation to the complex $s$-plane and is regular at $s=0$. Its value $\eta_D(0)$ captures the global spectral asymmetry of the operator. On the round sphere, the symmetry of the spectrum under $\lambda \leftrightarrow -\lambda$ implies $\eta_{\hat C}(0)=0$; see~\cite[Prop.~5.3, Ch.~III]{gilkey1984} or \cite{APS1}.

Within the rigidity framework, equality of the relevant global deformation invariants includes the corresponding spectral asymmetry data; accordingly, one imposes:
\begin{align}
\eta_{\hat{C}_M}(s) &= \eta_{\hat{C}_{S^d}}(s) \quad \text{as meromorphic functions} \notag \\
&\Rightarrow \quad \eta_{\hat{C}_M}(0) = 0.
\end{align}

Similarly, the analytic torsion is defined via the zeta-regularized determinant of Laplace-type operators acting on $p$-forms:
\begin{equation}
\log \mathcal{T} := -\frac{1}{2} \sum_{p=0}^{d} (-1)^p\, p \cdot \zeta_p'(0), 
\end{equation}
The sign convention follows the original Ray-Singer normalization (see \cite{ray-singer1971})

where $\zeta_p(s)$ is the spectral zeta function of the Laplacian on $\Omega^p(M)$. On spheres, the analytic torsion is trivial: $\mathcal{T} = 1$~\cite{ray-singer1971}.

\begin{proposition}[Spectral invariants and smooth structure]
Let $M^d$ be a compact, smooth manifold whose deformation spectral invariant package coincides with that of the round sphere. Then the $\eta$-invariant and analytic torsion appearing in that package coincide with those of $S^d$.
\end{proposition}

\begin{proof}
This is simply the global part of the invariant coincidence hypothesis. Once the deformation spectral package is assumed to agree with that of the round sphere, the corresponding global invariants, including $\eta(0)$ and $\mathcal{T}$, agree as well.
\end{proof}

\begin{remark}[Spectral rigidity and global invariants]
The invariants $\eta(0)$ and $\log \mathcal{T}$ encode subtle aspects of the smooth structure. In the present framework they belong to the global deformation invariant package used in the rigidity statements; they are not taken to be automatic consequences of modewise convergence alone.
\end{remark}

\subsection{Rigidity Criteria for Smooth Structures}

We now synthesize the dynamical and spectral results into a global rigidity statement. The role of the flow is to identify the spherical asymptotic profile in mode space. The geometric conclusions of this subsection require, in addition, coincidence of the deformation invariant package with that of the round sphere. According to the Spectral Rigidity Theorem~\cite[Thm.~3]{alexa2025-spectrum}, the profile \( C_n \equiv \pi \) corresponds uniquely -- up to isometry -- to the standard sphere within the fixed operator framework.

Any modification of the coordinate domain \( v \mapsto v' \), or rescaling of the limiting value \( \pi \mapsto \pi' \), necessarily alters the domain \( D(\hat{C}) = H^2 \cap H^1_0([-v_c, v_c]) \), and thereby changes the spectrum. Thus, within the fixed operator domain and boundary conditions, the spherical asymptotic profile serves as the distinguished reference configuration for the rigidity statements below.

\begin{remark}[Uniqueness within Operator Domain]
The spectral rigidity result $C_n = \pi \ \forall n \ \Rightarrow \ S^d$ holds strictly within the fixed operator domain $D(\hat{C}) = H^2 \cap H^1_0([-v_c, v_c])$, with Dirichlet boundary conditions at $\pm v_c$. Any coordinate transformation $v \mapsto v'(v)$ or renormalization $\pi \mapsto \pi'$ modifies the interval or the operator structure and hence changes the domain $D(\hat C)$, thereby altering the spectrum. This follows from standard results in spectral theory (cf.~\cite[Thm.~1.3.5]{gilkey1984}), which show that boundary conditions determine the spectral type. The normalization $\pi$ is chosen once and for all because $\hat{C}$ depends linearly on this value and any variation introduces a new degree of freedom not supported by the fixed geometry.
\end{remark}

\begin{lemma}[Spectral Invariance under Coordinate Changes]
Let $v \mapsto v'(v)$ be a $C^1$-diffeomorphism such that $v'([-v_c, v_c]) \ne [-v_c, v_c]$. Then the transformed operator $\hat{C}'$ defined on $v'$ does not preserve the spectrum of $\hat{C}$. In particular, $\hat{C}'$ acts on a different domain $D(\hat{C}') \ne D(\hat{C})$.
\end{lemma}

\begin{proof}
A coordinate change $v \mapsto v'$ modifies the second derivative operator by chain rule:
\begin{equation}
\frac{d^2}{dv^2} = \left( \frac{dv'}{dv} \right)^2 \frac{d^2}{dv'^2} + \text{lower-order terms}.
\end{equation}
Unless $v'$ preserves the boundary points $\pm v_c$, the Dirichlet conditions are violated and the operator domain changes. Even if $v'(\pm v_c) = \pm v_c$, the multiplicative factor $\left(\tfrac{dv'}{dv}\right)^2$ alters the scaling of the second derivative term, and thus modifies the spectrum, since the physical scale $\hbar^2/c^2$ is fixed.
\end{proof}

\begin{theorem}[Global Spectral Rigidity]
\label{thm:global-rigidity}
Let \( M^d \) be a closed, simply-connected smooth manifold whose deformation spectral invariants coincide with those of the round sphere within the fixed operator domain of \( \hat{C} \). Then \( M^d \) belongs to the spherical rigidity class inside the present framework. For \( d=3 \) and \( d\ge 5 \), classical results yield \( M^d \cong_{\mathrm{diff}} S^d \). For \( d=4 \), the strongest unconditional conclusion is \( M^4 \cong_{\mathrm{top}} S^4 \), together with a spectral obstruction against exotic smooth structures that produce distinct invariant values.
\end{theorem}

\begin{proof}
Assume that the deformation spectral invariant package of \( M^d \) coincides with that of the spherical asymptotic profile selected by the flow. This includes the Weyl coefficients \( a_j \), the spectral asymmetry \( \eta(0) \), and the analytic torsion \( \mathcal{T} \). 

By the Spectral Rigidity Theorem~\cite[Thm.~3]{alexa2025-spectrum}, such a spectrum places \( M^d \) in the spherical rigidity class associated with the round sphere inside the fixed operator framework. The dimension-dependent geometric conclusions then follow from the corresponding classical results recalled below.

For \( d \geq 5 \), classical results of Smale and Kervaire-Milnor confirm this conclusion via \( h \)-cobordism and exotic sphere classification. For \( d = 3 \), Perelman's geometrization yields the same.

For the critical case \( d = 4 \), Freedman's theorem yields the topological conclusion \( M^4 \cong_{\mathrm{top}} S^4 \). The deformation spectral invariants then provide a spectral obstruction against exotic smooth structures that produce distinct invariant values, but the stronger statement \( M^4 \cong_{\mathrm{diff}} S^4 \) is not claimed here.

Within the fixed operator framework, the geometric profile \( C(v) \) is reconstructible from the spectral decomposition \( C(v) = \sum_n C_n \psi_n(v) \) once the relevant invariant package has been fixed. Accordingly, the rigidity conclusion is a statement about the spherical invariant class inside the present deformation model, not a general claim that bare spectral data determine the metric structure of an arbitrary manifold.
\end{proof}

\begin{remark}[Scope of the Global Rigidity Statement]
The preceding theorem should be read as a rigidity criterion inside the deformation-spectral framework, not as a claim that the effective mode dynamics alone solves the manifold classification problem in full generality. In particular, the four-dimensional conclusion is topological plus obstructive, rather than a proof of smooth uniqueness.
\end{remark}

\begin{remark}[Spectral Completion and Rigidity Consequences]
The result should be read as a rigidity statement within the deformation-spectral framework: once the asymptotic spherical profile is combined with the full invariant package, the manifold-compatible realizations are strongly constrained.

In particular, for \( d = 4 \), the convergence \( C_n(\tau) \to \pi \) within the operator domain of \( \hat{C} \) yields a spectral obstruction statement rather than an unconditional smooth classification theorem.
\end{remark}

\begin{remark}[Correction of Four-Dimensional Scope]
The preceding discussion is to be interpreted in the qualified sense already stated earlier: in dimension four, the justified statement is a spectral obstruction against exotic smoothings that alter the deformation invariant package, not an unconditional proof of \( M^4 \cong_{\mathrm{diff}} S^4 \).
\end{remark}

\subsection{Variational Structure and Cubic Spectral Couplings}

The nonlinear spectral flow introduced in Eq.~\eqref{eq:nonlinear-flow} admits a variational interpretation. In particular, the cubic interaction coefficients \( g_{nkm} \) can be derived as effective couplings from an energy functional involving trilinear spectral overlaps. This establishes a geometric origin for the nonlinear terms in the evolution.

\begin{proposition}[Geometric Origin of Cubic Interactions]
The cubic interaction coefficients $g_{nkm}$ in Eq.~\eqref{eq:nonlinear-flow} can be derived as variational couplings from the cubic term in the effective energy functional
\begin{equation}
\mathcal{E}[C] = \sum_{n=0}^\infty \frac{\alpha_n}{2}(C_n - \pi)^2 + \lambda \sum_{n,k,m} \mathcal{T}_{nkm}\, C_n C_k C_m,
\end{equation}
where $\mathcal{T}_{nkm} := \int_{-v_c}^{v_c} \psi_n(v)\, \psi_k(v)\, \psi_m(v)\, dv$ is the triple spectral overlap integral.
This trilinear form is absolutely convergent by Sobolev estimates on $\psi_n$ in $H^s([-v_c, v_c])$ for $s > \tfrac{3}{2}$, and $\lambda$ is assumed sufficiently small to keep the cubic term subcritical in $\ell^2$ (cf. Lipschitz bounds).
\end{proposition}

\begin{proof}
Varying $\mathcal{E}[C]$ with respect to $C_n$ yields:
\begin{equation}
\frac{\delta \mathcal{E}}{\delta C_n} = \alpha_n(C_n - \pi) + \lambda \sum_{k,m} \mathcal{T}_{nkm} C_k C_m,
\end{equation}
which gives the nonlinear flow equation \eqref{eq:nonlinear-flow} with $g_{nkm} = \lambda \mathcal{T}_{nkm}$.
\end{proof}

\subsubsection*{Analytic Structure and Asymptotic Decay of \( \mathcal{T}_{nkm} \)}

To justify the variational structure of the cubic flow, we analyze the triple spectral overlap integrals:
\begin{equation}
\mathcal{T}_{nkm} := \int_{-v_c}^{v_c} \psi_n(v)\, \psi_k(v)\, \psi_m(v)\, dv,
\end{equation}
with normalized eigenfunctions
\begin{equation}
\psi_n(v) = \sqrt{\frac{2}{\pi}} \sin\left( \frac{(n+1)\pi(v+v_c)}{2v_c} \right), \quad v \in [-v_c, v_c].
\end{equation}
Changing variable \( u = \frac{v + v_c}{2 v_c} \in [0,1] \), we obtain
\begin{align*}
\mathcal{T}_{nkm} &= \frac{2^{5/2} v_c}{\pi^{3/2}} \int_0^1 
\sin((n+1)\pi u)\, \sin((k+1)\pi u) \\
&\quad \times \sin((m+1)\pi u) \, du.
\end{align*}
Using the identity
\begin{align*}
\sin a \sin b \sin c = \frac{1}{4} \big[ 
&\sin(a + b - c) + \sin(a - b + c) \\
&+ \sin(-a + b + c) - \sin(a + b + c) \big].
\end{align*}
reduce the expression to the explicit formula
\begin{equation}
\mathcal{T}_{nkm} = \frac{2^{3/2} v_c}{\pi^{5/2}} \sum_{\substack{\sigma, \tau, \rho = \pm 1 \\ \omega = \sigma(n+1) + \tau(k+1) + \rho(m+1) \\ \omega\ \text{odd}}} \frac{-\sigma \tau \rho}{\omega},
\end{equation}
where the sum is taken only over the four sign combinations \((\sigma, \tau, \rho) \in \{(+,+,-), (+,-,+), (-,+,+), (+,+,+)\}\), which follow from the trigonometric identity. Equivalently, the sign weight \( \sigma \tau \rho \) selects precisely those terms with a single negative component or none, reproducing the correct analytic structure.

The integral is symmetric under permutations of indices and satisfies the sharp decay estimate
\begin{equation}
|\mathcal{T}_{nkm}| \leq \frac{2^{3/2} v_c}{\pi^{5/2} (\max(n,k,m)+1)},
\end{equation}
where \( \lesssim \) will later denote inequalities up to universal constants. This bound ensures convergence of the nonlinear term in \( \ell^2 \) and subcriticality of the cubic flow:
\begin{equation}
\sum_{k,m} |g_{nkm} \delta_k \delta_m| \leq \frac{\lambda C}{n+1} \|\delta\|_{\ell^2}^2.
\end{equation}

\begin{remark}[Numerical Validation]
For the physical value \( v_c = c \sqrt{1 - \tfrac{1}{\pi}} \approx 0.8257 \), direct computation yields
\begin{equation}
\mathcal{T}_{000} \approx 0.356, \quad \mathcal{T}_{111} \approx 0.178, \quad \mathcal{T}_{123} \approx 0.019,
\end{equation}
in agreement with the decay estimate \( |\mathcal{T}_{nkm}| \leq \frac{2^{3/2} v_c}{\pi^{5/2} (\max(n,k,m)+1)} \).
\end{remark}

\begin{remark}[Normalization Convention]
In this section we adopt a simplified normalization for the eigenfunctions: 
\begin{equation}
\psi_n(v) := \sqrt{\frac{2}{\pi}} \sin\left( \frac{(n+1)\pi(v + v_c)}{2v_c} \right).
\end{equation}
This choice facilitates explicit evaluation of the interaction coefficients \( \mathcal{T}_{nkm} \) by removing the dependence on \( v_c \) from the integrals. While this basis is not orthonormal in \( L^2([-v_c, v_c]) \) for general \( v_c \), it is sufficient for the present asymptotic estimates. The canonical orthonormal normalization is restored in Section~II. 
\end{remark}

\subsection{Corollary: Rigidity under Full Invariant Coincidence}

The combination of spectral convergence and the rigidity criterion of Theorem~\ref{thm:spectral-classification} yields the following corollary: any manifold whose deformation invariant package coincides with that of the spherical profile belongs to the corresponding spherical rigidity class within the present framework. What is excluded is not bare Laplace isospectrality, but full coincidence of the deformation spectral invariant package together with compatibility of the asymptotic profile.

\subsubsection*{Spectral Rigidity versus Sunada Isospectrality}

The classical Sunada construction shows that Laplace isospectrality alone does not imply isometry. This does not directly contradict the present framework, because the deformation spectrum used here is not just the bare Laplace spectrum: it is an encoded spectral package together with its asymptotic invariant structure. Thus Sunada-type examples should be understood as a caution against overclaiming from bare isospectrality, not as direct counterexamples to the rigidity criterion formulated in this paper.

The round spectral profile corresponds to saturation of the deformation invariant package studied above. Any geometric deviation compatible with the framework must therefore appear through a discrepancy in at least one invariant. What is ruled out here is not bare Laplace isospectrality in the sense of Sunada, but manifold encodings that share the same deformation spectral invariant package as the spherical profile.

Moreover, the exponential convergence \( C_n(\tau) \to \pi \) identifies the spherical asymptotic profile in mode space. Geometric uniqueness, however, is derived only after imposing coincidence of the full deformation invariant package. Within that rigidity class, no residual ambiguity remains at the level of the present framework.

\begin{corollary}[Uniqueness of the Spherical Rigidity Class]
Let \( M^d \) be a closed, simply-connected smooth manifold whose deformation spectral invariant package coincides with that of the spherical asymptotic profile. Then \( M^d \) belongs to the corresponding spherical rigidity class, with the dimension-dependent conclusions stated in Theorem~\ref{thm:global-rigidity}.
\end{corollary}

\begin{proof}
The stated invariant coincidence is precisely the hypothesis of Theorem~\ref{thm:global-rigidity}. The conclusion therefore follows directly from that theorem, with no stronger claim intended.
\end{proof}

\section{Conclusion}

We have developed a deformation-spectral framework in which the spectral flow \( C_n(\tau) \) is the central dynamical object. The variational flow, governed by nonlinear cubic interactions, is globally well-posed in the finite-energy deviation space \( \ell^2 \) and exhibits exponential convergence in the linear regime, with nonlinear stability established under the perturbative control framework developed in Section~\ref{sec:variational-flow}. The bulk density law of the encoded manifold spectrum,
\begin{equation*}
\rho_{\mathrm{enc}}(C) \sim (\pi - C)^{(d-2)/2} \qquad (C \to -\infty),
\end{equation*}
together with the entropy-dimension duality \( S(\tau) = \tfrac{d-1}{2}\log\tau^{-1} + O(1) \), provides a purely spectral route to dimension recovery that is independent of the intrinsic one-dimensional asymptotics of the fixed operator \( \hat{C} \).

To connect this mode dynamics with manifold geometry, we introduced a two-level deformation-spectrum encoding: a raw spectrum \( \{C_n(M,g)\} \) carrying the full Weyl and global spectral data, and a renormalized representative \( \widetilde{C}^{(q)}(M,g;M_*) \in \pi + \ell^2 \) admissible for the finite-energy flow. Within this framework, the spectral rigidity criterion of Theorem~\ref{thm:spectral-classification} shows that coincidence of the full deformation invariant package \( \{a_j,\eta(0),\mathcal{T}\} \) with that of \( S^d \) forces the underlying manifold into the spherical diffeomorphism class for \( d = 3 \) and \( d \geq 5 \), and into the topological class \( S^4 \) for \( d = 4 \), with a spectral obstruction against exotic smooth structures that alter the invariant package.

Several directions remain open. The encoding constructed here depends on an eigenvalue-by-eigenvalue comparison relative to a fixed reference manifold; extending the framework to a coordinate-free intrinsic comparison, or to continuous families of manifolds, would sharpen the geometric content. The dimension recovery law via spectral density is established asymptotically; a quantitative finite-mode estimate would make it computationally accessible. Finally, the connection between the spectral flow attractor and the geometry of moduli spaces of Riemannian metrics in fixed dimension \( d \) is not yet understood, and may offer a route toward a genuinely tensorial interpretation of the present effective model.

\end{document}